\documentclass[12pt]{article}

\usepackage[margin=1in]{geometry}% <--- 1 in margin
\usepackage{setspace}
\makeatletter

\usepackage{amsmath,amssymb,amsthm}
\usepackage{textcomp}
\usepackage{graphicx}
\usepackage{tabularx} 
\usepackage{setspace} 
\usepackage{tikz} %for drawing graph with nodes
\usetikzlibrary{graphs}
\usepackage{float}
\usepackage{amsfonts} 
\usepackage[dvipsnames]{xcolor} %For using text color

\usepackage[colorlinks = true,linkcolor = blue,urlcolor  = blue,citecolor = blue,anchorcolor = blue]{hyperref}

\def\E{\mathbb{E}}
\def\P{\mathbb{P}}
\numberwithin{equation}{section}  % If you number theorems, etc. within sections,
\newtheorem{thm}{Theorem}[section]
\newtheorem{cor}[thm]{Corollary}
\newtheorem{lem}[thm]{Lemma}

\newtheorem{prop}[thm]{Proposition}
\newtheorem{defn}[thm]{Definition}

\title{Fluid limit of a Distributed Ledger Model with Random Delay} % insert title
\author{J. Feng and C. King\\
\\
Department of Mathematics \\
Northeastern University \\
MA 02115, USA}
\date{}
\begin{document}%\recd{}{}%Do not alter this line.
\maketitle

\begin{abstract}
Distributed ledgers, including blockchain and other decentralized databases, are designed to store information online where all trusted network members can update the data with transparency. The dynamics of ledger's development can be mathematically represented by a directed acyclic graph (DAG). In this paper, we study a DAG model which considers batch arrivals and random delay of attachment. We analyze the asymptotic behavior of this model by letting the arrival rate go to infinity and the inter-arrival time go to zero. We establish that the number of leaves in the DAG and various random variables characterizing the vertices in the DAG can be approximated by its fluid limit, represented as the solution to a set of delayed partial differential equations. Furthermore, we establish the stable state of this fluid limit and validate our findings through simulations.
\null \hspace{8cm} Secondary 60G46, 05C80
\end{abstract}

\section{Introduction}
A distributed ledger (DL) is a decentralized database where transactions are stored on a directed acyclic graph (DAG). The goal of any DL is to provide a secure and consistent record of transactions. Due to the widespread adoption of {DL technologies} for cryptocurrencies, there has been growing interest recently in formally establishing properties of the ledger.

{In the DAG associated with a DL, each vertex represents a block of information.} {A vertex arrives when a user begins creating a new block locally by collecting data. However, this vertex cannot be attached to the existing DAG until the user completes a proof-of-work (POW) computation \cite{back2002hashcash}. Only after solving POW can the block be broadcast to the network and added as a new vertex in the graph, with edges connecting it to one or more existing blocks according to a specified attachment rule. This time delay between arrival and attachment, imposed by the POW requirement, complicates the system's dynamics.} 

To better understand POW, consider the following description: When a new vertex $A$ arrives, one or more existing vertices in the DAG will be chosen, which are called the parents of $A$. Using the information in $A$ and its parents, a problem is then generated, which is typically implemented with a hash function as described in Section 3 in \cite{back2002hashcash}. The user trying to publish this vertex $A$ will have to solve the problem and this process is called POW. Vertex $A$ is accepted into the DAG only after the POW is completed, with edges connecting it to its chosen parents. The solution to the POW problem is stored in $A$, ensuring that any alteration to the data in $A$ or its parents changes the problem. As a result, the stored solution becomes invalid, enabling users to detect tampered data. This mechanism helps users to verify the data in the ledger and protect its record from being doctored. Once a vertex is attached by any future vertex, the transactions stored in the corresponding block are considered verified since at least one POW is finished to secure the data. Figure \ref{kchoice_demon_graph1} provides an example of a DAG.

\begin{figure}
\centering
\begin{tikzpicture}

    \node (0) [rectangle,draw] at (0,-4) {0};
    \node (1) [rectangle,draw] at (1,-4) {1};
    \node (2) [rectangle,draw] at (2,-4) {2};
    \node (3) [rectangle,draw] at (3,-4) {3};
    \node (4) [rectangle,draw,dashed] at (4,-4) {4};
    \node (5) [rectangle,draw,dashed] at (5,-4) {5};

    \graph{
        (1)->[bend left=60](0),
        (2)->[bend right=60](1),
        (2)->[bend right=60](0),
        (3)->[bend left=60](1),
        (4)->[bend right=60,dashed](2),
        (4)->[bend right=60,dashed](3),
        (5)->[bend left=60,dashed](2),
        (5)->[bend left=60,dashed](3)
    
    };
\end{tikzpicture}
\caption{Example of a DAG. A solid directed edge implies that the associated POW has been completed and the data has been accepted into the ledger. For example, vertex 3 has selected 1 as its parent and finished its POW. A dashed vertex with outgoing dashed edge implies the POW has started but not yet been finished. For example, vertex 5 has selected 2 and 3 as its parents but its POW has not yet been finished. We generally refer to the solid vertices that have no solid edge pointing toward them as \textit{tips}, which means the vertices have been accepted to the distributed ledger but has not yet been attached by any other vertices. For example, vertices $2$ and $3$ are tips because there is no solid edge toward them, while vertices $0$ and $1$ are not tips because they have been attached by other solid vertices. Also, vertices $4$ and $5$ are not considered as tips because they have not been accepted to the system by the fact that their POW have not been finished.}
\label{kchoice_demon_graph1}
\end{figure}

In 2008, blockchain \cite{Naka08,Zh17} was introduced as the first fully decentralized DL. In blockchain, each arriving block only attaches to one existing block that is the furthest from the root in the DAG, which is known as the longest chain protocol, hence leading to a linear chain structure. When there are multiple blocks with ongoing POWs, only one of them will be accepted into the system to maintain the chain structure and prevent double spending. Such a protocol requires a significant amount of computational power for POW and hence normally requires a transaction fee to compensate participants who contribute computational resources to solving the POW. 

The first version of IOTA \cite{Popov16} was an alternative to blockchain where the chain structure is no longer required and hence all honest blocks can be accepted into the system as long as their POWs are completed. Without the chain structure constraint, IOTA and other similar protocols are more challenging to analyze, particularly when accounting for POW delays. Tangle 2.0 \cite{muller2022tangle} was later introduced as a fully decentralized version of IOTA. This version includes a voting mechanism which replaces POW for achieving consensus. However, in Tangle 2.0, POW still exists in the sense that blocks are created using hash functions as described in Section IV.A in \cite{muller2022tangle}.

Distributed ledger technologies have evolved significantly in recent years and continue to be a dynamic and rapidly growing field. As a result, there is increasing interest in rigorously establishing the mathematical properties related to their security and stability. Recent research has explored several key aspects of distributed ledgers. One aspect is known as the consensus, where all users should agree on
the same set of non-conflicting information. This aspect has been extensively analyzed for blockchain designs, for example \cite{Pa19,Zh20,Am22}. In particular, the work \cite{PS20} identifies the relationship between consensus of DL and the one-endedness property of the corresponding DAG model for blockchain. Later work \cite{GO20} extends the analysis to other protocols and shows that IOTA-like DLs are one-ended when considering communication delay, with \cite{feng2023almost} generalizing the results for delays caused by both POW and communication. The work in \cite{frolkova2019bitcoin} investigates how random communication delay affects broadcasting newly created blocks to all users using a queuing model. When an adversarial power performs a double spending attack, the probability bound on the success rate of the attack is studied in \cite{Go19,Ga22}. Variants of blockchain have also been introduced to improve certain performance measures \cite{LY15,SS21}. The convergence of the long-time average tip count in individual data storage is studied in \cite{MU23} with its converging limit derived through simulation.

In this paper, we focus on the IOTA system, where each vertex selects two parents randomly from the tips with equal probability. As previously mentioned, a vertex becomes verified when it is attached by a future vertex. A critical aspect of the security of distributed ledgers is the time it takes for an accepted vertex to be attached by others, which is influenced by the number of tips. A fluid limit approximation for the dynamics of the number of tips, assuming a fixed POW duration, was introduced in \cite{Ck21}. This approximation provides key insights into the stationary number of tips and its convergence rate, which is critical for understanding the verification speed in the ledger. Other applications of the fluid limit introduced in \cite{Ck21} include the study of the behavior of vertices with conflicting information \cite{CK18} and the modification of the IOTA attachment mechanism \cite{Fer19}.

As an extension of \cite{Ck21}, this paper focuses on analyzing the fluid limit under the assumption that the POW duration is multinoulli distributed. When assuming a fixed duration of POW, the time when a tip ceases to be a tip is determined by the first time this tip gets selected as a parent. However, when POW duration is random, it becomes uncertain when a tip will cease to be a tip until it is actually attached by others. Therefore, various challenges arise when trying to analyze the dynamics of IOTA with random POW duration. In this paper, we introduce a new model to describe the dynamics by allowing that multiple vertices can arrive simultaneously. A new expression of the fluid limit arises because of this approach and the consideration of random duration of POW, which is proved to be a good approximation for the random dynamics of the model. In addition to the number of tips, which was analyzed in \cite{Ck21}, we also further distinguish the tips into different types according to their expected duration before they cease to be a tip. Such consideration of additional variables not only enables us to prove that the fluid limit is a good estimation of the random model but also provides insight into the proportion of tips that are at different levels of risk from a security aspect.

The remainder of this paper is organized as follows. Section \ref{kchoice_preliminary} provides a description of the model and the distribution of the random variables. Section \ref{kchoice_section_initial} introduces the fluid limit and discusses the initial conditions that affect the analysis. Our main results are summarized in Section \ref{k_mainresult}, which also presents several applications of the fluid limit model, demonstrating its efficacy in capturing key aspects of the DL's dynamics. Section \ref{k_summaryofproof} offers a high-level intuition for the proof of the main result, with the detailed proof provided in Sections \ref{kchoice_main_proof} and \ref{k_proofoflemmas}.

\section{Model set up}\label{kchoice_preliminary}
The DAG model used in this paper is a variation of the DAG model introduced in \cite{Ck21}. The main novelties of the model in this paper include the consideration of random POW durations as well as simultaneous arrivals of multiple vertices at each time. Before defining the model in detail, Section \ref{sec_notation} provides a brief introduction to the notation, indicating where each variable is formally defined. Then, we introduce variables that appear in the context of a DL. The distributions of the random variables and the evolution equations for the model will be given in Section \ref{kchoice_distribution} and in Section \ref{kchoice_section_updaterule} respectively.

\subsection{Notations}\label{sec_notation}
\textbf{Section \ref{sec_preliminaries}:} $\epsilon$ denotes the time between arrivals, and $N$ denotes the number of vertices arriving at each time $t_n=n\epsilon$. $\lambda=N/\epsilon$ denotes the arrival rate of vertices. The values $h_1<...<h_M$ represent the possible durations of POW and $N_i(t_n)$ denotes the number of $Type\ i$ arrivals at $t_n$. The variable $L(t_n)$ denotes the number of tips at $t_n$ while $F(t_n)$ and $W(t_n)$ denote the number of free tips and the number of pending tips respectively. The variable $F_i(t_n)$ counts the number of free tips selected as parents by $Type\ i$ arrivals but not by $Type\ j$ arrivals for any $j<i$. The number of $Type\ i$ pending tips at time $t_n$ with residual lifetime (RLT) $u$ is recorded by $W_i(t_n,u)$. The value $J_{i,j}(t_n,u)$ counts the number of $Type\ i$ pending tips at time $t_n$ with RLT $u$ that change to $Type\ j$ pending tips. Sometimes we will use $F_\lambda(t)$ and $L_\lambda(t)$ to denote $F(t)/\lambda$ and $L(t)/\lambda$ respectively for simplicity of presentation.

\textbf{Section \ref{k_section_fl}:} The values $l,f,w,w_i$ represent the fluid limit of variables $L,F,W,W_i$.

\textbf{Section \ref{k_mainresult}:} $g(t)$ records the difference between the discrete process and its fluid limits, with $T$ representing the time horizon. The values $\gamma,\xi,m$ are parameters describing the fluid limit, which are given by equations (\ref{k_fluildbound1}), (\ref{k_fluildbound2}) and (\ref{k_fluildbound3}).

\subsection{Preliminaries}\label{sec_preliminaries}
Let $\epsilon$ and $N$ be two constants and $t_n=n\epsilon$ for $n\in\mathbb{N}$, and we assume that $N$ vertices arrive simultaneously at each time $t_n$ with $\epsilon$ representing the inter-arrival time. We call a vertex that arrives at time $t_n$ an \textit{arrival at $t_n$}. We define $\lambda:=N/\epsilon$ which represents the arrival rate of the vertices. 

We assume that each new arriving vertex independently chooses its duration of POW from the set $\{h_1,h_2,...,h_M\}\subseteq\mathbb{R}$ with probability $p_1,p_2,\ldots,p_M$ respectively. Without loss of generality, we assume that $h_1<h_2<\ldots<h_M$, which are not necessarily equidistant, and that any $h_i$ is an integer multiple of $\epsilon$. Note that while arriving processes are commonly analyzed with $\epsilon = 1$ where all time variables are normalized by the inter-arrival time, we take $\epsilon \rightarrow 0$ so that we can use fixed $\{h_i\}_{i=1,\ldots,M}$ to represent real-world POW durations (e.g., minutes). The alternative approach of fixing $\epsilon = 1$ and letting $\{h_i\}_{i=1,\ldots,M} \rightarrow \infty$ would be less natural for interpreting results in terms of actual time scales. This current setting captures scenarios where many blocks arrive rapidly but each requires substantial but bounded computation time. Throughout the paper, we call an arriving vertex with POW duration $h_i$ a \textit{$Type\ i$ arrival} and the corresponding POW a \textit{$Type\ i$ POW}. We define $N_i(t_n)$ to be the number of {$Type\ i$} arrivals that arrive at time $t_n$, whose distribution will be provided in equation (\ref{kchoice_distribution_Ni}). Note that $\sum_{i=1}^{M}N_i(t_n)=N$ because at each time we have $N$ arriving vertices. Figure \ref{kchoice_demon_graph2} provides a graphical demonstration of the dynamics of the DAG used in this paper. It is worth noting that the assumption of a deterministic arrival process can be generalized to a binomial arrival process, a classical model closely related to Poisson arrivals. This can be done by assuming that each arrival is independently discarded with probability $p_0$. While this generalization requires minor modifications to the proof, the main arguments remain intact, as will become clearer in Section \ref{k_section_fl}.

\begin{figure}
\centering
\begin{tikzpicture}

    \node(0) at (-8,1) {...};
    \node (-31) [rectangle,draw] at (-6,2) {$A_1$};
    \node (-32) [rectangle,draw] at (-6,0) {$A_2$};
    \node (-21) [rectangle,draw] at (-4,2) {$B_1$};
    \node (-22) [rectangle,draw] at (-4,0) {$B_2$};
    \node (-11) [rectangle,draw] at (-2,2) {$C_1$};
    \node (-12) [rectangle,draw] at (-2,0) {$C_2$};
    \node (01) [rectangle,draw,dashed] at (0,2) {$D_1$};
    \node (02) [rectangle,draw,dashed] at (0,0) {$D_2$};
    \graph{
        {(-31),(-32),(-21),(-22)}->(0),
        (-11)->[bend left=25](-31),
        (-12)->(-31),
        (-12)->[bend right=10](0),
        (01)->[dashed](-32),
        (02)->[dashed](-21),
        (02)->[bend right=20,dashed](-32)
    };
\end{tikzpicture}
\caption{A demonstration of a DAG modeling IOTA at some time $t$. Each time there are $N=2$ arriving vertices, for example, $C_1,C_2$ arrived at the same time while $D_1,D_2$ arrived simultaneously $\epsilon$ time after $C_1,C_2$ arrived. A solid vertex with outgoing solid edge(s) implies that its corresponding POW has been completed and the vertex has been accepted into the system. For example, vertex $C_1$ has selected $A_1$ as its parent and finished its POW, hence it has been accepted into the system. A dashed vertex with outgoing dashed edge(s) implies the POW has not yet been finished. For example, vertex $D_2$ has selected $B_1$ and $A_2$ as its parents but the POW corresponding to $D_2$ has not yet been finished. The dashed vertices are the ones that have not been accepted into the system because their POWs are still in process. A dashed vertex will be included in the system once its corresponding POW is finished, then the dashed vertex and the dashed edge(s) originating from it will become solid.}
\label{kchoice_demon_graph2}
\end{figure}

We recall the terminology for tips and pending tips used in \cite{Ck21}. As in Figure \ref{kchoice_demon_graph2}, the solid vertices without any solid edge toward them are called \textit{tips}, which represent the vertices that are accepted into the system and have not yet been attached by other vertices. For example, the tips in Figure \ref{kchoice_demon_graph2} are $A_2,B_1,B_2,C_1,C_2$. At time $t_n$ we let $L(t_n)$ denote the number of tips in the system. Among these we distinguish the pending tips and the free tips. A tip is a \textit{pending tip} at time $t_n$ if it has been selected as a parent by some vertices that arrived at some time $t_j$ with $j < n$. A pending tip appears as a solid vertex with a dashed edge toward it, for example, the pending tips in Figure \ref{kchoice_demon_graph2} are vertices $B_1,A_2$. A tip is a \textit{free tip} at time $t_n$ if it is not a pending tip. A free tip appears as a solid vertex without any dashed edge toward it, for example, the free tips in Figure \ref{kchoice_demon_graph2} are $B_2,C_1,C_2$. We define $W(t_n)$ to be the number of pending tips at time $t_n$ and $F(t_n)$ to be the number of free tips at time $t_n$. It follows that $L(t_n) = W(t_n) + F(t_n)$. In order to analyze the evolution of random variables $L(t_n)$ and $F(t_n)$ which will be described in equations {(\ref{kchoice_evo_L})} and (\ref{kchoice_evo_F}) respectively, we first introduce some other random variables.

Essentially all distributed ledger protocols require each new vertex to have at least one parent. Vertices in blockchain typically have one parent, resulting in a linear graph. In addition to IOTA, there are other protocols where each vertex has multiple parents including the work in \cite{sompolinsky2016spectre,So21PHANTOM}. While this paper assumes that each vertex has two parents, the same procedure can be applied to protocols that require $k$ parents.

To model the parent selection algorithm used in IOTA, we assume that each vertex arriving at time $t_n$ chooses two parents with replacement from the set of tips, whose population is recorded by $L(t_n)$, and the tips are selected as parents with equal probability. The exact probability distribution will be given in equation (\ref{kchoice_distribution_parentselection}). Note that it could be the case that the vertex makes the same choice for both selections which results in a situation where the vertex has only one parent. In fact, whether we consider replacement does not affect our results, as the difference only appears in lower-order terms. For example, equation (\ref{eq_expected1}) considers only the highest-order term, which suffices for our analysis. The selection of parents affects the dynamics of the system, and we assume that the effect resulting from the parent selection will be reflected at the next step. For example, after a free tip is selected as a parent at $t_n$, it will become a pending tip at $t_{n+1}$.

In order to analyze the dynamics of the DAG, we now introduce the variables used to record the number of free tips selected at each step. We first provide the definitions of these variables and then give an example for demonstration. We suppose that the {$Type\ 1$} arrivals first make their selections, and we define $F_1(t_n)$ to be the number of free tips selected by {$Type\ 1$} arrivals at time $t_n$. Subsequently, the {$Type\ 2$} arrivals make their selections, and we define $F_2(t_n)$ to be the number of free tips selected by {$Type\ 2$} arrivals but not by any {$Type\ 1$} arrivals at time $t_n$. We then use the same idea to define $F_i(t_n)$ as the number of free tips that are selected by {$Type\ i$} arrivals but not by any {$Type\ j$} arrivals for all $j<i$. An example is given in Figure \ref{kchoice_demon_graph3}. To model the situation where a free tip becomes a pending tip after it is selected as a parent, we require that all the tips counted in $\sum_{i=1}^M F_i(t_n)$ become pending tips at the next step, that is, they are counted as part of $W(t_{n+1})$. Based on the description of $F_i(t_n)$ we have the distribution of $F_i(t_n)$ as provided in equations (\ref{kchoice_distribution_F1}), (\ref{kchoice_distribution_Fi}) and (\ref{kchoice_distribution_Fi1}).

\begin{figure}
\centering
\begin{tikzpicture}
    \node(-0) at (-7,2) {...};
    \node (31) [rectangle,draw] at (-6,3.5) {$A_1$};
    \node (32) [rectangle,draw] at (-6,2) {$A_2$}; 
    \node (33) [rectangle,draw] at (-6,1) {$A_3$}; 
    \node (21) [rectangle,draw] at (-5,3.5) {$B_1$};
    \node (22) [rectangle,draw] at (-5,2) {$B_2$};
    \node (23) [rectangle,draw] at (-5,1) {$B_3$};    
    \node (11) [rectangle,draw] at (-4,3.5) {$C_1$};
    \node (12) [rectangle,draw] at (-4,2) {$C_2$};
    \node (13) [rectangle,draw,dashed] at (-4,1) {$C_3$};    

    \node (41) at (-5,0) {$t_k$};
    \node (42) at (2,0) {$t_{k+1}$};
    \draw  (-3,4) -- (-3,0);

    \node(0) at (-2,2) {...};
    \node (-31) [rectangle,draw] at (-1,3.5) {$A_1$};
    \node (-32) [rectangle,draw] at (-1,2) {$A_2$}; 
    \node (-33) [rectangle,draw] at (-1,1) {$A_3$}; 
    \node (-21) [rectangle,draw] at (1,3.5) {$B_1$};
    \node (-22) [rectangle,draw] at (1,2) {$B_2$};
    \node (-23) [rectangle,draw] at (1,1) {$B_3$};    
    \node (-11) [rectangle,draw] at (3,3.5) {$C_1$};
    \node (-12) [rectangle,draw] at (3,2) {$C_2$};
    \node (-13) [rectangle,draw,dashed] at (3,1) {$C_3$};    
    \node (01) [rectangle,draw,dashed] at (5,3.5) {$D_1$};
    \node (02) [rectangle,draw,dashed] at (5,2) {$D_2$};
    \node (03) [rectangle,draw,dashed] at (5,1) {$D_3$};
    \graph{
        {(-32),(-33),(-31)}->(0);
        {(-22),(-21)}->(-32); (-23)->(-33); (-13)->[dashed](-23); {(-12),(-11)}->(-22); 
        (01)->[dashed, bend left=25]{(-31),(-21)};
        (02)->[dashed]{(-21),(-12)};
        (03)->[dashed,bend left=25]{(-12),(-23)};
         {(32),(33),(31)}->(-0);
        {(22),(21)}->(32); (23)->(33); (13)->[dashed](23); {(12),(11)}->(22); 
        
    };
\end{tikzpicture}
\caption{Example of the random variables $F_i(t_k)$. In this example, $N=3$ which means there are 3 arrivals at each time. At some time $t_k$ the vertices $D_1,D_2,D_3$ arrive and they have POW durations $h_1,h_2,h_3$ respectively. The graph on the left represents the graph at $t_k$ when $D_1,D_2,D_3$ arrive while the right graph represents the graph at $t_{k+1}$. At time $t_k$, the vertices $A_1,B_1,C_1,C_2$ are free tips since their corresponding POWs have been finished but they have not yet been selected as parents. The free tip $A_1$ is selected as a parent by a {$Type\ 1$} arrival $D_1$. The free tip $B_1$ is selected by a {$Type\ 1$} arrival $D_1$ and a {$Type\ 2$} arrival $D_2$. The free tip $C_2$ is selected by a {$Type\ 2$} arrival $D_2$ and a {$Type\ 3$} arrival $D_3$. Then $F_1(t_k)=2$ because both $A_1,B_1$ are selected by a {$Type\ 1$} arrival. Because $B_1$ is already included in $F_1(t_k)$, we have $F_2(t_k)=1$ although both $B_1,C_2$ are selected by a {$Type\ 2$} arrival. Finally, $F_3(t_k)=0$ because $C_2$ is the only free tip that is selected by a {$Type\ 3$} arrival but $C_2$ has been counted in $F_2(t_k)$. The selection will be reflected at $t_{k+1}$ and hence $A_1,B_1,C_2$ become pending tips at $t_{k+1}$.}
\label{kchoice_demon_graph3}
\end{figure}

Recall that a POW with duration $h_i$ is called a {$Type\ i$} POW and a pending tip is a tip that has been selected as a parent. If vertex $A$ is selected as a parent of a vertex $B$, we say that the POW corresponding to $B$ is \textit{directed to} vertex $A$. A pending tip will cease to be a tip when a POW directed to it is finished. We define the {\em residual lifetime (RLT)} of a pending tip at time $t_n$ as the remaining time (measured in units of $\epsilon$) required for finishing one of the ongoing POWs directed to the pending tip at time $t_n$. For example, a pending tip $A$ at time $t_7$ might have been selected by three vertices: an arrival at time $t_4$ with corresponding POW duration $h_2=7\epsilon$ whose POW will be finished at time $t_4+h_2=t_{11}$; an arrival at time $t_5$ with corresponding POW duration $h_3=8\epsilon$ whose POW will be finished at $t_5+h_3=t_{13}$; an arrival at time $t_6$ with corresponding POW duration $h_5=10\epsilon$ whose POW will be finished at $t_6+h_5=t_{16}$. Since the three listed POWs will be finished at $t_{11}$, $t_{13}$ and $t_{16}$ respectively, the RLT of this pending tip $A$ at time $t_7$ is $(11-7)\epsilon=4\epsilon$ which is determined by the POW with duration $h_2=7\epsilon$. Suppose no arrivals at time $t_7$ select $A$ as a parent, then at time $t_8$, the RLT of the pending tip $A$ at time $t_8$ is $4\epsilon-\epsilon=3\epsilon$ because $\epsilon$ time has passed. However, notice that the RLT at $t_n$ only uses information up to time $t_n$ and it can jump to a much lower value if the pending tip is selected as a parent by a new arrival. Using the same example, suppose that at time $t_8$ an arrival with POW duration $h_1=2\epsilon$ selects $A$ as its parent. Then at time $t_9$ there are four POWs that are directed to $A$ and the POW that starts at $t_8$ will finish the earliest at time $t_8+2\epsilon=t_{10}$ and hence the RLT of $A$ at $t_9$ is $\epsilon$. When the RLT of a pending tip becomes zero at $t_n$, the pending tip will cease to be a tip at the next step, that is, the vertex will no longer be included in the set of tips at time $t_{n+1}$.

Notice that there could be multiple POWs that determine the RLT of a pending tip at a given time. In this set of vertices whose corresponding POWs determine the RLT of a given pending tip at a given time, the vertices might be of different \textit{Types}, i.e. they have different durations of POWs. For example, a pending tip $A$ at time $t_9$ was selected by two vertices: a {$Type\ 2$} arrival at time $t_6$ with POW duration $h_2=4\epsilon$ which will be finished at $t_{10}$; a {$Type\ 1$} arrival at $t_8$ with POW duration $h_1=2\epsilon$ which will be finished at $t_{10}$. Then at time $t_9$, the RLT of $A$ is $\epsilon$ which is determined by both the {$Type\ 1$} POW and the {$Type\ 2$} POW directed to $A$. If the smallest Type of the POWs that determine the RLT of the pending tip $A$ at time $t_n$ is $i$, then we call the pending tip $A$ a \textit{$Type\ i$ pending tip}. In the example, the smallest Type of POW that determines the RLT of $A$ at $t_9$ is {$Type\ 1$} and hence $A$ is a {$Type\ 1$} pending tip at $t_9$. Note that a pending tip can still be selected as a parent by a new arrival before one of the POWs directed to it has been completed, so a pending tip may change from {$Type\ i$} to {$Type\ j$} for $i>j$. However notice that a pending tip cannot change from {$Type\ i$} to {$Type\ j$} if $i<j$ because being selected by a vertex with higher Type will not affect the related set of POWs that determine the RLT of the pending tip.

By the above definitions, we denote $\mathcal{W}_i(t_n, u)$ to be the set of {$Type\ i$} pending tips with RLT equal to $u$ at time $t_n$. And we denote $W_i(t_n, u):=|\mathcal{W}_i(t_n, u)|$ as the number of {$Type\ i$} pending tips at time $t_n$ with RLT equal to $u$. For a {$Type\ i$} pending tip counted in $W_i(t_n,u)$, it might jump to a lower {$Type\ j$} where $h_j\leq u$ because it gets selected by a {$Type\ j$} arrival at time $t_n$. If this happens, the pending tip will be counted in $W_j(t_{n+1},h_j-\epsilon)$ as a $Type\ j$ pending tip in the next moment. We hence define $J_{i,j}(t_n,u)$ to be the number of pending tips in the set ${\cal W}_i(t_n, u)$ that jump to {$Type\ j$} at the next step. The distribution of $J_{i,j}(t_n,u)$ will be defined in equation (\ref{kchoice_distribution_J}).

\subsection{Distributions of random variables}\label{kchoice_distribution}
Based on the intuition of the random variables described in Section \ref{kchoice_preliminary}, we first introduce the distribution of the variables including $N_i(t_n)$, $F_i(t_n)$ and $J_{i,j}(t_n,u)$. Since there are $N$ arrivals at each time and each of them has probability $p_k$ of having POW duration $h_k$ for $k=1,...,M$, the term $N_i(t_n)$ satisfies a binomial distribution as it counts the number of arrivals at time $t_n$ that have POW duration $h_i$. This means 
\begin{align}
N_i(t_n)\sim \text{Binomial}(N,p_i) \text{ and } \sum_{i=1}^M N_i(t_n)=N.\label{kchoice_distribution_Ni}
\end{align}

The random variables $F_i(t_n)$ and $J_{i,j}(t_n,u)$ are generated by selections of parents. To model the parent selection algorithm used in IOTA, we assume that each vertex arriving at time $t_n$ chooses two parents with replacement from the set of tips whose population is recorded by $L(t_n)$ and the tips are selected as parents with equal probability. Hence for an arrival $v$ at time $t_n$ and a tip $v'$ at time $t_n$, 
\begin{align} 
\P(v' \text{ is not selected as a parent of }v|L(t_n)):= \left(1-\frac{1}{L(t_n)} \right)^2. \label{kchoice_distribution_parentselection}
\end{align} 

The term $F_i(t_n)$ calculates how many free tips are selected by the arrivals with duration $h_i$ but are not selected by any arrival with POW duration less than $h_i$ at time $t_n$. We suppose that the $N$ arrivals first choose their duration for POW and then select their parents. By equation (\ref{kchoice_distribution_parentselection}) and that the number of arrivals with POW duration $h_1$ is recorded by $N_1(t_n)$, each free tip has probability $(1-1/L(t_n))^{2N_1(t_n)}$ of not getting selected by any {$Type\ 1$} arrivals at $t_n$. We use binary variables $X_{n,1,j}$ with $j=1,2,...,F(t_n)$ to denote whether each free tip is selected as a parent by any {$Type\ 1$} arrival at $t_n$, where $X_{n,1,j}=1$ if the corresponding free tip is selected. The conditional probability that a free tip is not selected as a parent by any {$Type\ 1$} arrival at $t_n$ is:
\begin{align*} 
\P(X_{n,1,j}=0|L(t_n),N_1(t_n),F(t_n) ):=\left(1 - \frac{1}{L(t_n)} \right)^{2 N_1(t_n)} \text{ for } j=1,2,...,F(t_n).
\end{align*}
Therefore we define
\begin{align}
F_1(t_n) :=\sum_{j=1}^{F(t_n)} X_{n,1,j}.  \label{kchoice_distribution_F1}
\end{align}
where the random variables $X_{n,1,1},X_{n,1,2},\cdots,X_{n,1,F(t_n)}$ are dependent. 

By the same idea as well as the assumption that the {$Type\ i$} arrivals select their parents after the {$Type\ i-1$} arrivals, we have that by the time the {$Type\ i$} arrivals select their parents, the number of tips that are possible to be counted toward $F_i(t_n)$ is $F(t_n)-\sum _{k=1}^{i-1}F_k(t_n)$. Given $n$ and $i$, we use $X_{n,i,j}$ with $j=1,...,F(t_n)-\sum_{k=1}^{i-1}F_k(t_n)$ to denote whether each free tip that is possible to be counted toward $F_i(t_n)$ is selected by a {$Type\ i$} arrival, with $X_{n,i,j}=1$ if the free tip is selected and $X_{n,i,j}=0$ otherwise: 
\begin{align} 
\P(X_{n,i,j}=0|L(t_n),F(t_n),N_i(t_n), F_k(t_n):k< i):=\left(1 - \frac{1}{L(t_n)}\right)^{2 N_i(t_n)}. \label{kchoice_distribution_Fi}
\end{align}
As a generalization of equation (\ref{kchoice_distribution_F1}), we define
\begin{align}
F_i(t_n) := \sum_{j=1}^{F(t_n)-\sum_{k=1}^{i-1}F_k(t_n)} X_{n,i,j} \text{ for } i=2,3,...,M,\label{kchoice_distribution_Fi1}
\end{align}
where $M$ is the total number of possible durations of POW.

Recall that $W_i(t_n,u)$ denotes the number of {$Type\ i$} pending tips whose RLT is $u$. For $i>j$ and $u\geq h_j$, $J_{i,j}(t_n, u)$ counts how many pending tips with RLT $u$ jump from {$Type\ i$} to {$Type\ j$} because they are selected by a {$Type\ j$} arrival. This number is also generated by random selection carried out by the $N$ arrivals at time $t_n$ and we use $X'_{n,i,u,j,k}$ to denote whether each {$Type\ i$} pending tip with RLT $u$ is selected by a {$Type\ j$} arrival at $t_n$:
\begin{align*}
\P(X'_{n,i,u,j,k}=0|W_i(t_n, u),N_j(t_n)):=\left(1 - \frac{1}{L(t_n)}\right)^{2 N_{j}(t_n)},
\end{align*}
where $k=1,2,...,W_i(t_n, u)$. We then define
\begin{align}
J_{i,j}(t_n, u)  :=  \sum_{k=1}^{W_i(t_n, u)} X'_{n,i,u,j,k}.\label{kchoice_distribution_J}
\end{align}

\subsection{Update rules of the Model}\label{kchoice_section_updaterule}
In this paper, our primary focus is to analyze the dynamics of the random variables\\$\{F(t_n),L(t_n), W_{i}(t_n,u):u\in[0,h_i],\epsilon\text{ divides } u\}$. While the initial condition will be introduced in Section \ref{kchoice_section_initial}, we now describe how the variables including $L(t_n), F(t_n)$ and $W_i(t_n,u)$ are updated at each time step. First, consider the number of free tips $F(t_n)$.
This increases because some vertices complete their POWs, and it decreases because some free tips are selected
as parents. Therefore we define:
\begin{align}
F(t_{n+1}) - F(t_n) = \sum_{i=1}^{M} N_i(t_n-h_i)-\sum_{i=1}^{M}F_i(t_n).\label{kchoice_evo_F}
\end{align}

Next consider the number of tips $L(t_n)$ which increases because some vertices complete their POWs and decreases because some tips are attached by other vertices and cease to be tips. To model this mechanism, we define
\begin{align}
L(t_{n+1}) - L(t_n) &= \sum_{i=1}^M N_i(t_n-h_i) -\sum_{i=1}^M F_i(t_n-h_i)\nonumber\\
&-\sum_{j<i}\sum_{u=h_j+\epsilon}^{h_i-\epsilon}J_{i,j}(t_n-h_j,u)
+\sum_{j<i}\sum_{u=h_j+\epsilon}^{h_i-\epsilon}J_{i,j}(t_n-u,u)\label{kchoice_evo_L}
\end{align}
by considering the following scenarios: 1), $\sum_{i=1}^{M} N_i(t_n-h_i)$ counts how many vertices should finish their POWs at $t_n$ and become free tips at time $t_{n+1}$. 2), if a free tip was selected as a parent by a vertex whose POW will be finished at time $t_n$, then it might be the case that this free tip ceases to be a tip due to the completion of the POW directed to it. We use $\sum_{i=1}^M F_i(t_n-h_i)$ to consider this situation. 3), some of the {$Type\ i>1$} pending tips that are anticipated to finish POW after $t_{n}$ may have been selected by {$Type\ j<i$} arrivals at time $t_n-h_j$, thus changing to {$Type\ j$} pending tips and hence are validated at time $t_n$. Such consideration corresponds to the first double summation in the second line of equation (\ref{kchoice_evo_L}). 4), in scenarios 2) and 3), some of the {$Type\ i$} pending tips created at time $t_n - h_i$ may have been changed to {$Type\ j<i$} at some time before $t_n-h_i$, and thus have completed their POW prior to time $t_n$. These tips should be removed from the count of the reduction of $L(t_n)$ and are taken into account through the last term in equation (\ref{kchoice_evo_L}). Scenario 4) demonstrates how random POW durations complicate the model.

Next we consider $W_i(t_{n}, u)$, the number of {$Type\ i$} pending tips with RLT $u$ at time $t_n$. The value $W_i(t_{n+1}, h_i-\epsilon)$ consists of: 1), the free tips selected by {$Type\ i$} arrivals at $t_{n}$ and hence become {$Type\ i$} pending tips at time $t_{n+1}$; 2), the pending tips that just change to {$Type\ i$} at time $t_{n+1}$. Hence we define
\begin{align}
W_i(t_{n+1}, h_i-\epsilon) :=F_i(t_n)+\sum_{j>i}\sum_{u=h_i}^{h_j-\epsilon} J_{j,i}(t_n,u).\label{kchoice_evo_Wi1}
\end{align}
The {$Type\ i$} pending tips with RLT $u$ will have a lower RLT at the next step but some of them will change to {$Type\ j$} for some $j<i$. Hence we define 
\begin{align}
W_i(t_{n+1}, u - \epsilon) := W_i(t_{n}, u) - \sum_{h_j\leq u}J_{i,j}(t_n, u)\quad\text{for }{u\leq h_i-\epsilon}.\label{kchoice_evo_Wi2}
\end{align}
If a pending tip has RLT equal to zero, it will no longer be a tip at the next step and hence ${W}_i(t_n, u)=0$ for $u<0$. Furthermore, if the RLT of a {$Type\ i$} pending tip is less than 
$h_k$ then it cannot change to a {$Type\ k$} pending tip at any later time. In particular, if the RLT of a {$Type\ i$} pending tip is less than $h_1$ then it will always be {$Type\ i$} until it is no longer a tip. Therefore we define $W_i(t_n + s, u-s) = W_i(t_n, u) \quad \text{for all $u<h_1$, $s<u$}$.

The assumption of deterministic arrival and its generalization to binomial arrival lead to tractable update rules as described by equations (\ref{kchoice_evo_F}), (\ref{kchoice_evo_L}), (\ref{kchoice_evo_Wi1}) and (\ref{kchoice_evo_Wi2}), where each RHS of the update rules has finite terms. In contrast, Poisson arrivals introduce an unbounded number of possible arrivals in the update rules and continuous-time dynamics, both of which complicate the analysis.

Since a tip is either a pending tip or a free tip, we have $W(t_n)=L(t_n)-F(t_n)$ and hence knowing equations (\ref{kchoice_evo_F}), (\ref{kchoice_evo_L}), (\ref{kchoice_evo_Wi1}) and (\ref{kchoice_evo_Wi2}) is enough. However, for the purpose of studying the fluid limit that will be introduced in Section \ref{k_section_fl}, we introduce an alternative equation to calculate the number of pending tips, i.e., $W(t_n)$. Note that a pending tip must be a free tip at a prior time. This implies that a pending tip at $t_n$ must correspond to a free tip that is selected as a parent by a {$Type\ i$} arrival between $[t_n-h_i,t_n-\epsilon]$ for some $i$. Among the free tips that are selected as parents by any {$Type\ i$} arrivals between $[t_n-h_i,t_n-\epsilon]$ for some $i$, some of them might be selected again and get attached before or at time $t_n$. So we need to subtract the number of pending tips that have anticipated attachment time after $t_n$ but changed to pending tips with lower Type and got attached before $t_n$. These quantities that need to be subtracted are calculated by using $J_{i,j}$. To illustrate, we first give an example using $M=2$, i.e., there are two possible values of POW duration.

For the special case of $M=2$, the number of pending tips is calculated as
\begin{align*}
W(t_n) = \sum_{i=1}^2\sum_{s=\epsilon}^{h_i}F_i(t_n-s)-\sum_{s=h_1+\epsilon}^{h_2-\epsilon}\sum_{u=s}^{h_2-\epsilon}  J_{2,1}(t_n-s,u),
\end{align*} 
where the range of $J_{2,1}(t_n-s,u)$ should satisfy the following conditions: 1), we subtract $J_{2,1}(t_n-s,u)$ because the {$Type\ 2$} pending tips were anticipated to finish POW at time $\geq t_n$ but changed to {$Type\ 1$} and hence finished before $t_n$. So $t_n-s+u\geq t_n$ and $t_n-s+h_1<t_n$; 2), since the greatest possible RLT is $h_2-\epsilon$, we have $u\leq h_2-\epsilon$; 3), when $u=h_2-\epsilon$, by 1) we need $t_n-s+h_2-\epsilon\geq t_n$. Hence $s\leq h_2-\epsilon$.

For $M>1$, repeating the argument by generalizing the pair $(2,1)$ to $(i,j)$ yields:
\begin{align}
W(t_n) = \sum_{i=1}^M\sum_{s=\epsilon}^{h_i}F_i(t_n-s)-\sum_{i=2}^M\sum_{j=1}^{i-1}\sum_{s=h_j+\epsilon}^{h_i-\epsilon}\sum_{u=s}^{h_i-\epsilon} J_{i,j}(t_n-s,u) .\label{kchoice_evo_W}
\end{align}

\subsection{The Fluid Limit}\label{k_section_fl}
We are interested in the model considering the random variables $L(t_n)$, $F(t_n)$, $W_i(t_n,u)$ in the limit where $\lambda, N, \epsilon^{-1} \rightarrow \infty$. In particular, we will first make a reasonable conjecture for the fluid limit based on the expected values and then establish that the fluid limit is close to the random process. For the random variables, we focus on the leading order term in their expected values and assume that the remainder terms are negligible in the limit $\lambda, N, \epsilon^{-1} \rightarrow \infty$. 
Using the distributions from Sections \ref{sec_preliminaries} and \ref{kchoice_distribution}, we obtain the following leading order terms:
\begin{align}
N_i(t_n) = N  p_i + ...,\quad\quad F_i(t_n) = 2 \, N \, p_i \, \frac{F(t_n)}{L(t_n)} + ...,\nonumber \\
\text{and }J_{i,j}(t_n, u) = 2 \, N \, p_j \, \frac{W_{i}(t_n, u)}{L(t_n)} + \cdots \quad \text{for } i>j,\ h_j\leq u.\label{eq_expected1}
\end{align}
Note that two appears in equation \eqref{eq_expected1} and later derivations because two parents are selected for each vertex. If $k$ parents are selected for each vertex instead, two will be replaced by $k$.

We expect that there are deterministic functions $l(t),f(t),w_i(t,u)$ such that $l(t_n)\approx \lambda^{-1}  L(t_n)$, $f(t_n) \approx \lambda^{-1}  F(t_n)$ and $w_i(t_n, u) \approx N^{-1}  W_i(t_n, u)$. By the update rules in equations (\ref{kchoice_evo_F}), (\ref{kchoice_evo_L}), (\ref{kchoice_evo_Wi1}), (\ref{kchoice_evo_Wi2}) and (\ref{kchoice_evo_W}), taking expected values on the right-hand side (RHS) and only keeping the leading order term in each expected value yields:
\begin{align}
&f(t_{n+1}) - f(t_n) = \epsilon - 2 \, \epsilon \, \frac{f(t_n)}{l(t_n)} \nonumber,\\
&l(t_{n+1}) - l(t_n) = \epsilon - 2 \, \epsilon \sum_{i=1}^M p_i\frac{f(t_n-h_i)}{l(t_n-h_i)}  - 2 \, \epsilon \sum_{j<i}\sum_{u=h_j+\epsilon}^{h_i-\epsilon}\epsilon p_j\frac{w_i(t_n-h_i,u)}{l(t_n-h_i)} \nonumber \\
&\quad\quad\quad\quad\quad\quad\quad\quad + 2 \, \epsilon \sum_{j<i}\sum_{u=h_j+\epsilon}^{h_i-\epsilon} \epsilon p_j\frac{w_i(t_n-u,u)}{l(t_n-u)}\nonumber,\\
&w_i(t_{n+1}, u-\epsilon) - w_i(t_n, s) = - 2 \, \epsilon \,\sum_{h_j\leq u} p_j\frac{w_i(t_n,u)}{l(t_n)}  \nonumber,\\
&w_i(t_{n+1},h_i-\epsilon) = 2 \, p_i \, \frac{f(t_n)}{l(t_n)}+\sum_{j>i}\sum_{u= h_i}^{h_j-\epsilon}2p_i\frac{\epsilon w_j(t_n,u)}{l(t_n)}\nonumber,\\
&w(t_n) =\sum_{i=1}^{M}\sum_{s=\epsilon}^{h_i}2\epsilon p_i\frac{f(t_n-s)}{l(t_n-s)}-2\sum_{i>j}\sum_{s=h_j+\epsilon}^{h_i-\epsilon}\sum_{u=s}^{h_i-\epsilon} \epsilon^2p_j\frac{w_i(t_n-s,u)}{l(t_n-s)}.\label{kchoice_asym_0}
\end{align}

We now formally take the limit $\lambda, N, \epsilon^{-1} \rightarrow \infty$. The asymptotic equations (\ref{kchoice_asym_0}) lead to the following delayed partial differential equations (PDEs), and we call the corresponding solution to the PDEs the \textit{fluid limit} of the random DAG model:
\begin{align} 
&\frac{df}{dt}=1-2\frac{f(t)}{l(t)},\label{k_fl_f}\\
&\frac{dl}{dt}=1-2\sum_{i=1}^M p_i\frac{f(t-h_i)}{l(t-h_i)}-2\sum_{j<i}\int_{h_j}^{h_i}p_j\frac{w_i(t-h_i,u)}{l(t-h_i)}du\nonumber\\
&\quad\quad\quad+2\sum_{j<i}\int_{h_j}^{h_i}p_j\frac{w_i(t-u,u)}{l(t-u)}du,\label{k_fl_l}\\
&\left(\frac{\partial}{\partial t}-\frac{\partial}{\partial u}\right)w_i(t,u)=-2\sum_{h_j\leq u} p_j\frac{w_i(t,u)}{l(t)},\label{k_fl_wi1}\\
&w_i(t,h_i)=2 \, p_i \, \frac{f(t)}{l(t)}+\sum_{j>i}\int_{ h_i}^{h_j}2p_i\frac{w_j(t,u)}{l(t)}du, \label{k_fl_wi2}\\
&w(t) = 2\sum_{i=1}^{M}\int_{0}^{h_i} p_i\frac{f(t-s)}{l(t-s)}ds-2\sum_{i>j}\int_{h_j}^{h_i}\int_{s}^{h_i} p_j\frac{w_i(t-s,u)}{l(t-s)}duds.\label{k_fl_w}
\end{align}

As mentioned in Section \ref{sec_preliminaries}, the assumption of deterministic arrivals can be generalized to binomial arrivals. This generalization can be carried out by multiplying the expected values in equation \eqref{eq_expected1} by $1-p_0$, similar to how $p_i$ appears in the calculations. The ansatz for the fluid limit and the proof of its approximation quality in Section \ref{kchoice_main_proof} then proceed similarly.

In order to show the existence and uniqueness of the solution to the PDEs, we will first define the initial condition in Section \ref{kchoice_section_initial}.

\subsection{Initial Conditions}\label{kchoice_section_initial}
Both the random process of $\{F(t_n),L(t_n), W_{i}(t_n,u):u\in[0,h_i], \epsilon \text{ {divides} } u\}$ defined in equations (\ref{kchoice_evo_F}), (\ref{kchoice_evo_L}), (\ref{kchoice_evo_Wi1}) and (\ref{kchoice_evo_Wi2}) and their corresponding PDEs that describe the fluid limit should be supplemented with an initial condition. The initial condition of the random process, which can be found by observing the update rules, is defined as follows
\begin{align*}
&F(t_j)=x_j&j=-2h_M\epsilon^{-1},...,0,\\
&L(t_j)=y_j&j=-2h_M\epsilon^{-1},...,0,\\
&W_i(t_0,k\epsilon)=\omega_k   &k=1,2,..., h_i\epsilon^{-1}-1,\\
&N_i(t_j)=a_{i,j}  &j=-h_M\epsilon^{-1},-h_M\epsilon^{-1}+1,...,-1,\\
&F_i(t_j)=b_{i,j}&j=-h_M\epsilon^{-1},-h_M\epsilon^{-1}+1,...,-1,\\
&J_{i,r}(t_j,k\epsilon)=c_{i,r,j,k}   &j=-h_M\epsilon^{-1},-h_i\epsilon^{-1}+1,...,-1,\\
&&k=h_i/\epsilon-1,h_i/\epsilon-2,...,h_r/\epsilon.
\end{align*} 
Fixing the above values is sufficient to generate the future of the model using equations (\ref{kchoice_evo_F}), (\ref{kchoice_evo_L}), (\ref{kchoice_evo_Wi1}) and (\ref{kchoice_evo_Wi2}).

\begin{defn}
An initial condition for equations (\ref{kchoice_evo_F}), (\ref{kchoice_evo_L}), (\ref{kchoice_evo_Wi1}) and (\ref{kchoice_evo_Wi2}) is a\\
\underline{proper initial condition with respect to $\theta_1$} if the requirements stated in equations (\ref{kchoice_ini_eq1}), (\ref{kchoice_ini_eq2}), (\ref{kchoice_ini_eq3}), (\ref{kchoice_ini_eq4}), (\ref{kchoice_ini_eq5}), (\ref{kchoice_ini_eq7}) and (\ref{kchoice_ini_eq6}) hold for given $\theta_1$. 
\end{defn}

The requirements assumed in equations (\ref{kchoice_ini_eq1}), (\ref{kchoice_ini_eq2}), (\ref{kchoice_ini_eq3}), (\ref{kchoice_ini_eq4}), (\ref{kchoice_ini_eq5}), (\ref{kchoice_ini_eq7}) and (\ref{kchoice_ini_eq6}), which are stated in the lemmas where they are needed, are reasonable assumptions of initial conditions which require that some quantities in the initial condition are close to their expected values when they are generated. The validity of these conditions is stated in Proposition \ref{k_prop_ini}.

\begin{prop}\label{k_prop_ini}
Given any $\theta_1>0$, there exist $\lambda^*$ and $N^*$ such that for $\lambda>\lambda^*$ and $N>N^*$, the collection $\Pi(\theta_1)$ of proper initial conditions with respect to $\theta_1$ is nonempty.
\end{prop}
The proof of Proposition \ref{k_prop_ini} is equivalent to the proof of the parallel statements that arise within the intervals $t\geq0$ instead of $t\leq 0$. 
For illustration, see Lemmas \ref{kchoice_dn} and \ref{lem_Initial_condition}, where Lemma \ref{kchoice_dn} assumes the initial condition specified by equation (\ref{kchoice_ini_eq1}) and Lemma \ref{lem_Initial_condition} validates this assumption. The core idea is that the initial conditions are also randomly generated using the described dynamics, and therefore the properties that are established in $t\geq0$ can be applied to $t\leq0$ as well.

For the initial condition of the PDEs corresponding to the fluid limit, we first let $f(t_j)=F(t_j)/\lambda, l(t_j)=L(t_j)/\lambda$ for  $t_j=-2h_M,...,0$. One can then use interpolations like Lagrange polynomial or cubic spline interpolation to construct $f(t),l(t)$ for $t\in[-2h_M,0]$ so that both functions are continuously differentiable. By following equations (\ref{k_fl_wi1}) and (\ref{k_fl_wi2}), one can generate $w_i(t,u)$ for $t\in[-h_M,0]$ as part of the initial condition for the fluid limit.

Using the constructed initial condition for the fluid limit, we will use a similar approach as in \cite{Ck21} based on a method called the method of steps \cite{driver2012ordinary} to show existence and uniqueness of the fluid limit. We first consider $f,l,w_i$ in the interval $t\in[0,h_1]$. Note that $l(t)$ for $t\in[0,h_1]$ can be directly computed using equation (\ref{k_fl_l}) as the RHS only considers initial conditions. Together with equation (\ref{k_fl_f}), the solution for $f(t)$ within $t\in[0,h_1]$ can be calculated. $l(t),f(t)$ within $t\in[0,h_1]$ along with the initial conditions can then be used to solve $w_i(t,h_i)$ within $t\in[0,h_1]$ using equations (\ref{k_fl_wi1}) and (\ref{k_fl_wi2}). Using the initial conditions and the solutions $f,l,w_i$ within the interval $t\in[0,h_1]$, repeating the same procedure leads to the solution in $t\in[h_1,2h_1]$ and so on iteratively.

\section{Main Results}\label{k_mainresult}
We now state the main result of this paper, which establishes that the scaled random process is close to its fluid limit under suitable conditions.
\begin{thm}\label{kchoice_main_theorem}
Let $[0,T]$ denote a fixed time interval with $T$ being the time horizon. Let $\delta,\gamma, \xi,m>0$. Define 
\begin{align} 
g(t)=\sup_{0<s\leq t}\left\{\left|\frac{F(s)}{\lambda}-f(s)\right|+\left|\frac{L(s)}{\lambda}-l(s)\right|+\sum_{i=1}^M \sup_{u\leq h_i} \left|\frac{W_i(s,u)}{N}-w_i(s,u)\right| \right\},\label{kchoice_difference}
\end{align}
and $n_T=\epsilon^{-1}T$ (assuming it is an integer for simplicity). Then there exist $\alpha,\beta>0$ independent of $\lambda,N,\epsilon$ such that
\begin{align} 
\P(g(T)>\delta)\leq& \beta\epsilon^{-1}\exp\left(-\beta \frac{\theta^2}{T\epsilon} \right)+\beta T\epsilon^{-1}\exp\left(-\beta\lambda\right)\nonumber\\
&+\beta \exp\left(-\beta\theta^2N\right)+\beta T\frac{1}{\epsilon^3N\theta^2}+\beta T\frac{1}{\epsilon^3N}\exp\left(-\beta\theta^2\epsilon^2N\right),\label{eq_main_result}
\end{align}   
if the following conditions hold: 1), the initial condition for the random process as in Section \ref{kchoice_section_initial} is a proper initial condition with respect to $\theta=(\delta e^{-\alpha T})/(3\alpha)$; 2), the fluid limit within the interval $[0,T]$ and its initial condition satisfy equations (\ref{k_fluildbound1}), (\ref{k_fluildbound2}) and (\ref{k_fluildbound3}) given by
\begin{align} 
\left|\frac{\partial}{\partial u}\frac{w_i(t-u,u)}{l(t-u)}\right|&\leq \gamma  \quad \text{ for } t\in[0,T],u\in(0,h_M),\label{k_fluildbound1}\\
\frac{w_i(t,u)}{l(t)}&\leq \xi \quad \text{ for }  t\geq [-h_M,T], u\in[0,h_i], \label{k_fluildbound2}\\
l(t)&\geq m \quad\quad \forall t\in[-h_M,T).\label{k_fluildbound3}
\end{align} 
\end{thm}

While the complete proof of Theorem \ref{kchoice_main_theorem} will be provided in Section \ref{kchoice_main_proof}, we will also provide an introduction of the main idea of proof in Section \ref{k_summaryofproof}.

Equation (\ref{eq_main_result}) provides a probability bound for understanding the difference $g(T)$ between the random process and its fluid limit. It implies that the difference goes to zero in probability provided that $\epsilon^{-1},\lambda,N,\epsilon^3N\rightarrow \infty$. Since $\lambda=N/\epsilon$, equation (\ref{eq_main_result}) implies that it suffices for $N$ to go to infinity faster than $\epsilon^{-3}$. In addition, the difference $g(T)$ goes to zero in probability for any finite time horizon $T$, but may diverge if $T$ also goes to infinity. Our parametrization with $\epsilon \rightarrow 0$ and fixed $\{h_i\}$ makes this result interpretable in terms of real-world time scales. In contrast, with $\epsilon=1$ and $T, \{h_i\}_{i=1,\ldots,M} \rightarrow \infty$, time is measured in units of inter-arrival intervals, making it difficult to interpret what constitutes a "fixed time horizon" in natural time units.

By Theorem \ref{kchoice_main_theorem}, one can predict the evolution of the random process by either using the fluid limit which is the solution to a set of delayed partial differential equations or using a simulation result of the random process. Notice that for simulation, instead of actually performing the POW which contributes to almost all the computational power consumed in the system, it suffices to generate a random duration of POW. Figure \ref{k_simulation} provides an example.

\begin{figure}
    \centering
    \includegraphics[width=0.48\textwidth]{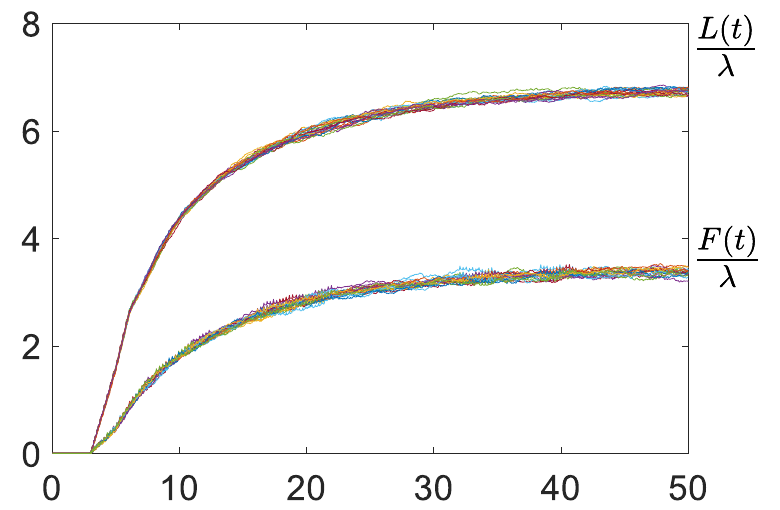}
    \includegraphics[width=0.48\textwidth]{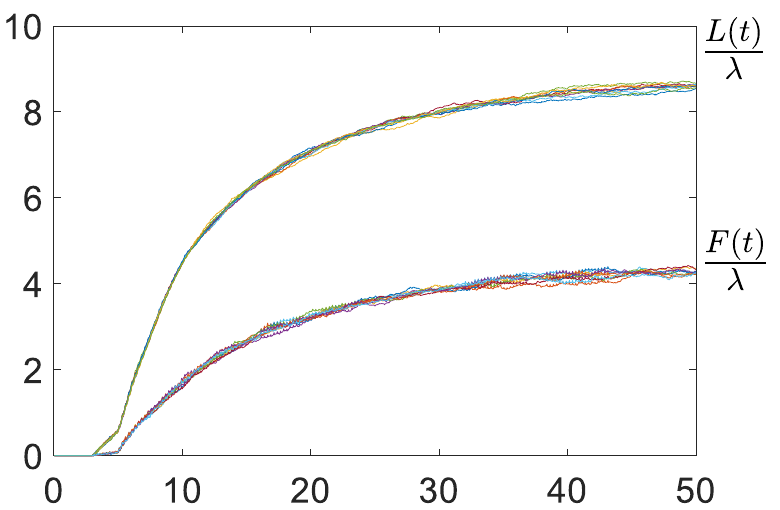}
    \caption{Simulations with $\lambda=400, N=20, \epsilon=0.05$ where the vertical axis displays values for $L(t)/\lambda,F(t)/\lambda$ and the horizontal axis represents time $t_n$. $M$ is assumed to be two corresponding to Proposition \ref{k_equili_2}. The left figure corresponds to the simulation with $p_1=0.8$ and $p_2=0.2$ while the right figure corresponds to the simulation with $p_1=0.3$ and $p_2=0.7$. Both simulations assume $h_1=3$ and $h_2=5$. Multiple simulations are performed with the same parameters and we can see that the scaled random processes behave almost deterministically with minor deviation.}
    \label{k_simulation}
\end{figure}

On the other hand, studying the equilibrium point of the fluid limit can provide useful insight into the random process. By assuming $f(t)=f$, $l(t)=l$, $w(t)=l-f$ and $w_i(t,u)=w_i(u)$, we can solve the fluid limit equations and derive the equilibrium point. Proposition \ref{k_equili_2} provides a result when $M$, the number of possible POW durations, is two.
\begin{prop}\label{k_equili_2}
    When $M=2$, the equilibrium point of equations (\ref{k_fl_f}), (\ref{k_fl_l}), (\ref{k_fl_wi1}) and (\ref{k_fl_wi2}) is given by:
    \begin{align*}
    w_1(t,u)&=p_1,f(t)=w(t)=l(t)/2,\\
    w_2(t,u)&=p_2\exp{\left(-\frac{2p_1}{l}(h_2-u)\right)}\quad\text{for } u\in[h_2-h_1,h_2],\\
    \left(1 - \frac{p_2}{p_1}\right)  l &= 2 h_1 - \frac{p_2}{p_1}  l  \exp\left(- \frac{2 p_1}{l}  (h_2 - h_1)\right).
    \end{align*}
\end{prop}

Proposition \ref{k_equili_2} and Fig.\ref{k_simulation} demonstrate that the parameters for POW durations including $p_i$ and $h_i$ determine the values of $f(t),w(t),l(t)$ in stable state where the property that $f(t)=w(t)=l(t)/2$ remains fixed. As demonstrated by Fig.\ref{k_simulation}, $f(t),l(t), w(t)$ in stable state have lower values when $\sum_{i=1}^Mh_ip_i$ is smaller. This is because tips are attached by future vertices at a higher rate when the expected value of POW duration $\sum_{i=1}^Mh_ip_i$ decreases, and therefore fewer tips will be waiting to be attached.

For the equilibrium point when $M>2$, $f(t)=w(t)=l(t)/2$ holds by equation (\ref{k_fl_f}) and $w_i(t,u)$ first decays exponentially and then remains constant as $u$ changes because of equation (\ref{k_fl_wi1}). It is worth noting that when $k$ parents are selected for each arriving vertex instead of two,
equation (\ref{eq_expected1}) implies that $f(t)=l(t)/k$ in the stable state. In addition, the stable state value $l$ in Proposition \ref{k_equili_2} coincides with the finding of \cite{Pen21} which studies the stable state of tips assuming the arrival of vertices follows a Poisson process.

\subsection{Discussion}
While Theorem \ref{kchoice_main_theorem} assumes that POW durations are drawn i.i.d. from a distribution with bounded support on $M$ finite values, where $M$ can be any finite number, the result can be extended to a distribution with unbounded support. One can construct a sequence of distributions with $M\rightarrow \infty$ so that the sequence converges to the unbounded distribution for POW durations. Each distribution with finite $M$ can be seen as a truncated version of the unbounded distribution. By showing that the random process with bounded POW durations can approximate the process with unbounded POW durations within the fixed time horizon $[0,T]$, Theorem \ref{kchoice_main_theorem} extends to unbounded distributions.

There are multiple future directions to extend this topic. For example, \cite{Ck21} has investigated how fast the fluid limit converges to a stable state when the POW has a fixed duration. Whether the convergence rate remains the same when POW is random is an interesting question. In addition, adversarial vertices are an important topic in distributed ledger systems, particularly in relation to security. This has been studied in \cite{SS22} for blockchain systems but has not yet been explored in the context of IOTA and other similar protocols. For example, \cite{dembo2020everything} and \cite{gobel2016bitcoin} investigate the security of blockchain under adversarial attack by including the consideration of communication delay. Because of the differences between the attachment rules for blockchain and DAG-based protocols, many results from blockchain systems cannot be directly applied to protocols with DAG representation. Therefore, conducting deeper investigations into IOTA and similar systems by considering adversarial attack and communication delay provides valuable insights into the performance differences between DAG-based protocols and blockchain.

\section{Idea of Proof}\label{k_summaryofproof}
We first introduce the main idea of the proof for Theorem \ref{kchoice_main_theorem} which has also been used in \cite{Ck21}. Consider a stochastic process $\{X(t_n)\}_n$ which depends on a parameter $\lambda$ with $t_n=n/\lambda$ and $X(t_n)\rightarrow \infty$ as $\lambda\rightarrow \infty$. The goal is to illustrate the idea of bounding $|X(t_n)/\lambda-x(t_n)|$ where $x(t)$ represents the solution to the differential equation as the fluid limit. We construct the initial condition of the fluid limit such that $x(t_0)=X(t_0)/\lambda$. We first use a telescoping sum to decompose the difference $|X(t_n)/\lambda-x(t_n)|$:
\begin{align*}
|\lambda^{-1} X(t_n)-x(t_n)|\leq\left|\lambda^{-1}\sum_{j=0}^{n-1} \left[ X(t_{j+1})-X(t_{j})-\lambda(x(t_{j+1})-x(t_{j})) \right]\right|.
\end{align*}
Denote $\Delta(X,j):= X(t_{j+1})-X(t_{j})$. Let $\mathcal{F}_n$ denote the  $\sigma$-algebra generated by all $X(t_k)$ up to time $t_n$. We then have
\begin{align}
&|\lambda^{-1} X(t_n)-x(t_n)|\leq \left|\lambda^{-1}\sum_{j=0}^{n-1} \left[ \Delta(X,j)-\lambda\int_{t_j}^{t_{j+1}}x'(s)ds\right]  \right|\label{kchoice_main_idea1}\\
&\leq \left|\lambda^{-1}\sum_{j=0}^{n-1} \left[ \Delta(X,j)-E[\Delta(X,j)|\mathcal{F}_{j}]+E[\Delta(X,j)|\mathcal{F}_{j}]-\lambda\int_{t_j}^{t_{j+1}}x'(s)ds\right]  \right|\nonumber\\
&\leq \left|\lambda^{-1}\sum_{j=0}^{n-1}  \Delta(X,j)-E[\Delta(X,j)|\mathcal{F}_{j}]\right|+
\left|\lambda^{-1}\sum_{j=0}^{n-1}\left[ E[\Delta(X,j)|\mathcal{F}_{j}]-\lambda\int_{t_j}^{t_{j+1}}x'(s)ds\right] \right|\nonumber.
\end{align}
The difference $|X(t_n)/\lambda-x(t_n)|$ has been separated into two terms, where the first absolute value represents the fluctuations and the second term represents the differences between expected values and their fluid limit. We will refer to such a technique as the \textit{separation technique}. We then apply probability inequalities such as Lemma \ref{kchoice_azuma} to get that for any $\theta>0$
\begin{align}
\P \left(\left|\lambda^{-1}\sum_{j=0}^{n-1}  \Delta(X,j)-E[\Delta(X,j)|\mathcal{F}_{j}]\right|\geq \theta \right)\leq o(\lambda^{-1})\label{kchoice_main_idea2}.
\end{align} 
For the second part in equation (\ref{kchoice_main_idea1}), with the construction of the fluid limit we can derive that
\begin{align} 
\left|\lambda^{-1}\sum_{j=0}^{n-1} E[\Delta(X,j)|\mathcal{F}_{j}]-\lambda\int_{t_j}^{t_{j+1}}x'(s)ds \right|\leq o(\lambda^{-1})\sum_{j=0}^{n-1} g(t_j)+o(\lambda^{-1}),\label{kchoice_main_idea3}
\end{align}
where $g(t)=\sup_{0\leq s\leq t}|\lambda^{-1}X(t)-x(t)|$. In this paper, the inequality (\ref{kchoice_main_idea3}) holds because the fluid limit is derived by dropping lower order terms in the conditional expectation and taking the limit, where the things we dropped are bounded by the term $o(\lambda^{-1})$ and the difference between the remaining terms and the fluid limit will be bounded by $o(\lambda^{-1})\sum_{j=0}^{n-1}g(t_j)$.

For $T$ denoting the time horizon, generalizing $|\lambda^{-1}X(t_n)-x(t_n)|$ to $|\lambda^{-1}X(t)-x(t)|$ yields, on some event that has probability going to one, 
\begin{align*}
g(T)\leq \theta+o(\lambda^{-1})\sum_{j=0}^{\lfloor T\lambda^{-1}\rfloor-1} g(t_j)+o(\lambda^{-1}).
\end{align*}
We then use discrete Gronwall's lemma with a suitable choice of $\theta$ to get the desired upper bound of $\P\left(g(T)>\delta\right)$.

\section{Proof of Theorem \ref{kchoice_main_theorem}}\label{kchoice_main_proof}
Throughout the discussion, let $T$ be the time horizon and, without loss of generality, denote the integer $n_T=T\epsilon^{-1}$. Also, we define $\sigma_n$ to be the $\sigma$-algebra generated by $N_i(t_k),F_i(t_k)$ for $k< n$ and $i\leq M$ as well as $J_{i,j}(t_k,u)$ for $k< n,j<i\leq M$ and $u\in[h_j,h_i)$, which corresponds to the parents selection and POW duration generated for all arrivals that arrive before $t_n$. Note that by the evolution equations (\ref{kchoice_evo_F}), (\ref{kchoice_evo_L}),(\ref{kchoice_evo_Wi1}) and (\ref{kchoice_evo_Wi2}), the values of $F(t_n)$, $L(t_n)$ and $W_i(t_n,u)$ are $\sigma_{n}$-measurable.

To prove Theorem \ref{kchoice_main_theorem}, we will first analyze the difference between the scaled number of free tips $F(t_n)/\lambda$ and its fluid limit $f(t_n)$, followed by generalizing the differences at discrete times to the continuous time interval $s\in [0,T]$. This will be done in Section \ref{k_df_freetips}. We will repeat this process for the other terms that appear in equation (\ref{kchoice_difference}) where Section \ref{k_df_tips} corresponds to the terms for the number of tips and Section \ref{k_df_pendingtips} corresponds to the numbers of pending tips. Finally, we will put them together and use Gronwall's lemma to establish the theorem in Section \ref{kchoice_finalstep}. 

We first introduce an important lemma that will be used repeatedly throughout the proof to provide probability bounds like Lemmas \ref{kchoice_dn} and \ref{kchoice_df1}.
\begin{lem}\label{kchoice_azuma}
Let $\{\sum_{i=1}^{n}Y_i\}_{n\in\mathbb{N}}$ be a martingale and $|Y_i|\leq 4N$, then by Azuma-Hoeffding inequality \cite{Williams91},
\begin{align*} 
\P\left\{\sup_{0\leq n\leq n_T}\left|\sum_{j=1}^{n}Y_j\right|\geq \theta\right\}\leq 2\exp\left(-\frac{\theta^2}{2(n_T)(16N^2)}\right)\\
\implies\P\left\{\sup_{0\leq n\leq n_T}\left|\lambda^{-1}\sum_{j=1}^{n}Y_j\right|\geq \theta\right\} 
\leq 2\exp\left(-\frac{\theta^2}{32T\epsilon}\right).
\end{align*}  
\end{lem}

\subsection{Difference in number of free tips}\label{k_df_freetips} 
In this section, we establish the upper bound of $|\lambda^{-1}F(t_n)-f(t_n)|$, which will be stated in equation (\ref{k_ef}). In order to do that, we will follow the idea introduced through equations (\ref{kchoice_main_idea1}), (\ref{kchoice_main_idea2}) and (\ref{kchoice_main_idea3}). From now on, we assume that $n>0$.
\begin{align}
&\lambda^{-1}F(t_n)-f(t_n)=\lambda^{-1}\sum_{j=0}^{n-1}F(t_{j+1})-F(t_j)-\lambda(f(t_{j+1})-f(t_j))\nonumber\\
=&\lambda^{-1}\sum_{j=0}^{n-1}\left(
\sum_{i=1}^{M} N_i(t_j-h_i)-\sum_{i=1}^{M}F_i(t_j)
-\lambda\int_{t_j}^{t_{j+1}} 1-2\frac{f(s)}{l(s)}ds
\right)\nonumber\\
=&\lambda^{-1}\sum_{j=0}^{n-1}\left(
-N+
\sum_{i=1}^{M} N_i(t_j-h_i)
\right)+\sum_{i=1}^{M}\lambda^{-1}\sum_{j=0}^{n-1}
\left(
E[F_i(t_j)|\sigma_j]-F_i(t_j)
\right)\nonumber\\
&+\sum_{i=1}^{M}\lambda^{-1}\sum_{j=0}^{n-1}
\left(
\lambda\int_{t_j}^{t_{j+1}} 2p_i\frac{f(s)}{l(s)}ds-E[F_i(t_j)|\sigma_j]
\right)\label{kchoice_decompose_F},
\end{align}
which is the result following the idea in equation (\ref{kchoice_main_idea1}). We now establish Lemmas \ref{kchoice_dn}, \ref{kchoice_df1}, \ref{kchoice_E3}, \ref{kchoice_df2} which provide probability bounds on each summation in equation (\ref{kchoice_decompose_F}). All the lemmas from here on will be proved in Section \ref{k_proofoflemmas}.

\begin{lem}\label{kchoice_dn} %{dn}
Let $D_n=\lambda^{-1}\sum_{j=0}^{n-1} \left(-N+\sum_{i=1}^M N_i(t_j-h_i)\right)$. Assume $\epsilon|h_i$ and let $n_i=h_i\epsilon^{-1}$. 
If for the initial condition we have for all $i=1,2,...,M$,
\begin{align}
\sup_{-n_i\leq k< 0}\left|\lambda^{-1}\sum_{j=-n_i}^{k} \left(N_i(t_j)-p_i N\right)\right|<\theta_1 \label{kchoice_ini_eq1}.
\end{align}
Then $\P(E_1):=\P(\sup_{0<n\leq n_T}|D_n|\geq M\theta+M\theta_1)\leq 2M\exp\left(-\frac{\theta^2}{32T\epsilon}\right).$
\end{lem}

The validity of the assumption in equation (\ref{kchoice_ini_eq1}) is given in Lemma \ref{lem_Initial_condition}, whose proof is analogous to the proof of Lemma \ref{kchoice_dn}. This similarity arises from the fact that both the random process after time $t=0$ and its initial condition are randomly generated by the same rules described in Section \ref{kchoice_distribution}. Based on the illustration using Lemma \ref{kchoice_dn} and Lemma \ref{lem_Initial_condition}, the lemmas and proofs of other assumptions on the initial conditions will be omitted.

\begin{lem}\label{lem_Initial_condition}
Suppose the initial condition of the random process described in Section \ref{kchoice_section_initial} is generated according to the distributions described in Section \ref{kchoice_distribution}. Then 
\begin{align*}
\P\left(\bigcup_{i=1}^{M}\left\{\sup_{-n_i\leq k<0}\left|\lambda^{-1}\sum_{j=-n_i}^{k} \left(N_i(t_j)-p_iN\right)\right|\geq\theta_1\right\}\right)\leq 2M\exp\left(-\frac{\theta_1^2}{32h_M\epsilon}\right),
\end{align*}
which vanishes for $\epsilon\rightarrow 0$.
\end{lem}

\begin{lem}\label{kchoice_df1}
For any $i$, we have
\begin{align*} 
\P(E_{2,i}):=\P\left\{\sup_{0< n\leq n_T}\left|\lambda^{-1}\sum_{j=0}^{n-1}  E[F_i(t_j)|\sigma_j]-F_i(t_j)\right|\geq \theta  \right\}\leq \exp\left(-\frac{\theta^2}{32T\epsilon}\right).
\end{align*} 
Hence $\P(E_2):=\P(\cup_{i=1}^M E_{2,i})\leq 2M\exp\left(-\frac{\theta^2}{32T\epsilon}\right)$. (Proved using Lemma \ref{kchoice_azuma})
\end{lem}
Lemma \ref{kchoice_df1} holds because $E[F_i(t_j)|\sigma_j]-F_i(t_j)$ is a random variable that satisfies the requirements in Lemma \ref{kchoice_azuma}.

\begin{lem}\label{kchoice_E3}%{temptwo}
If for the initial condition we have
\begin{align}
\inf_{-h_M\epsilon^{-1}\leq n\leq 3/2h_1\epsilon^{-1}}\sum_{j=n-2h_1\epsilon^{-1}}^{0} N_1(t_j)>1 \label{kchoice_ini_eq2},
\end{align}
\begin{align*} 
\text{then }\P(E_3):=\P\left(\inf_{-h_k\epsilon^{-1}\leq n\leq n_T}L_{\lambda}(t_n)\leq 1\right)\leq T\epsilon^{-1} \exp\left(-\frac{(1-p_1h_1/2)^2\lambda}{2h_1}\right).
\end{align*} 
\end{lem}
Lemma \ref{kchoice_E3} allows us to derive the following Lemma \ref{kchoice_df2}.

\begin{lem}\label{kchoice_df2}%{tempthree}
Let
\begin{align}
H_{i,j+1}=-E[F_i(t_j)|\sigma_j]+\lambda\int_{t_j}^{t_{j+1}} 2p_i\frac{f(s)}{l(s)}ds\label{def_H},
\end{align}
then, on event $E_3^C$, $\lambda^{-1}\sum_{j=0}^{n-1}|H_{i,j+1}|\leq 16\epsilon T+\sum_{j=0}^{n-1} \frac{2p_i\epsilon g(t_{j+1})}{m}$ for all $i$.
\end{lem}

Combining Lemmas \ref{kchoice_dn}, \ref{kchoice_df1}, \ref{kchoice_E3}, \ref{kchoice_df2} yields that, on the event $(\cup_{i=1}^3 E_i)^c$, 
\begin{align*} 
\left|\lambda^{-1}F(t_n)-f(t_n)\right|\leq 2M\theta+M\theta_1+16M\epsilon T+\sum_{j=0}^{n-1} \frac{2\epsilon g(t_{j+1})}{m}.
\end{align*} 
We can extend this inequality to $t\in[t_{n},t_{n+1})$ by defining that $F(t)=F(t_k)$ for $t\in[t_k,t_{k+1})$:
\begin{align} 
&\left|\lambda^{-1}F(t)-f(t)\right|\leq\left|\lambda^{-1}F(t_n)-f(t_n)\right|+\left|f(t_n)-f(t)\right|\nonumber\\
&\leq\left|\lambda^{-1}F(t_n)-f(t_n)\right|+\left|\int_{t_n}^{t}1-2\frac{f(s)}{l(s)}ds\right|\leq\left|\lambda^{-1}F(t_n)-f(t_n)\right|+\epsilon+2\epsilon\nonumber\\
&\leq3\epsilon+2M\theta+M\theta_1+16M\epsilon T+\sum_{j=0}^{n-1} \frac{2\epsilon g(t_{j+1})}{m}\label{k_ef},
\end{align} 
where we use the property that $f(s)\leq l(s)$.

\subsection{Difference in number of tips}\label{k_df_tips}
Next, by following a similar procedure as in Section \ref{k_df_freetips}, we establish equation (\ref{k_el}) that provides an upper bound of $\left|\lambda^{-1}L(t_n)-l(t_n)\right|$.
\begin{align}
&\lambda^{-1}L(t_n)-l(t_n)=\lambda^{-1}(L(t_n)-L(t_0))-(f(t_n)-f(t_0))\nonumber\\
=&\lambda^{-1}\sum_{j=0}^{n-1}
\left[
L(t_{j+1})-L(t_j)-\lambda\int_{t_j}^{t_{j+1}}f'(t)dt
\right]\nonumber\\
=&\lambda^{-1}\sum_{j=0}^{n-1}\left(
\sum_{i=1}^{M} N_i(t_j-h_i)-N\right)\nonumber\\
&+\lambda^{-1}\sum_{j=0}^{n-1}\left(
-\sum_{i=1}^M F_i(t_j-h_i)+2\lambda\sum_{i=1}^M p_i\int_{t_j}^{t_{j+1}}\frac{f(t-h_i)}{l(t-h_i)}dt
\right)\nonumber\\
&+\lambda^{-1}\sum_{j=0}^{n-1}\left[
2\lambda\sum_{r<i}\int_{t_j}^{t_{j+1}}\int_{h_r}^{h_i}p_r\frac{w_i(t-h_i,u)}{l(t-h_i)}dudt-
\sum_{r<i}\sum_{u=h_r+\epsilon}^{h_i-\epsilon}J_{i,r}(t_j-h_i,u)
\right]\nonumber\\
&+\lambda^{-1}\sum_{j=0}^{n-1}\left[
2\lambda\sum_{r<i}\int_{t_j}^{t_{j+1}}\int_{h_r}^{h_i}p_r\frac{w_i(t-u,u)}{l(t-u)}dudt-
\sum_{r<i}\sum_{u=h_r+\epsilon}^{h_i-\epsilon}J_{i,r}(t_j-u,u)
\right]\label{kchoice_decompose_L}
\end{align}
where the last equality is derived from equations (\ref{kchoice_evo_L}) and (\ref{k_fl_l}). In equation (\ref{kchoice_decompose_L}), we have decomposed the difference $L(t_n)/\lambda-l(t_n)$ into four parts and we will analyze each part on the RHS of equation (\ref{kchoice_decompose_L}) by subtracting and adding conditional expectations as mentioned in equation (\ref{kchoice_main_idea1}). 

The first part in equation (\ref{kchoice_decompose_L}) is
$\lambda^{-1}\sum_{j=0}^{n-1} \left(-N+\sum_{i=1}^M N_i(t_j-h_i)\right)$
which is already bounded using Lemma \ref{kchoice_dn}. The second part in equation (\ref{kchoice_decompose_L}) is
\begin{align*} 
&\lambda^{-1}\sum_{j=0}^{n-1}\left[\left(
2\lambda\sum_{i=1}^M p_i\int_{t_j}^{t_{j+1}}\frac{f(t-h_i)}{l(t-h_i)}dt
\right)
-\sum_{i=1}^M F_i(t_j-h_i)
\right]\\
=&\lambda^{-1}\sum_{i=1}^M \sum_{j=0}^{n-1}\left[\left(
2 p_i\lambda\int_{t_j}^{t_{j+1}}\frac{f(t-h_i)}{l(t-h_i)}dt
\right)-F_i(t_j-h_i)\right.\\
& -E[F_i(t_j-h_i)|\sigma_{j-n_i}] +E[F_i(t_j-h_i)|\sigma_{j-n_i}]
\Biggr],
\end{align*} 
whose summands are very similar to the objects in Lemma \ref{kchoice_df1} and Lemma \ref{kchoice_df2} except that the time is $t-h_i$ instead of $t$ and hence some of the terms are initial conditions. Therefore we first make some modifications based on Lemma \ref{kchoice_df1} to establish Corollary \ref{kchoice_df1_cor} and then we directly apply Lemma \ref{kchoice_df2} to get Corollary \ref{kchoice_df2_cor}.

\begin{cor}\label{kchoice_df1_cor}
By Lemma \ref{kchoice_df1}, if for any $i$ we have
\begin{align}
\sup_{0< n\leq n_i}\left|\lambda^{-1}\sum_{j=0}^{n-1}  E[F_i(t_j-h_i)|\sigma_j]-F_i(t_j-h_i)\right|\leq \theta_1,\label{kchoice_ini_eq3}
\end{align} 
then
{\footnotesize
\begin{align*}
&\P(E_{4,i}):=\P\left\{\sup_{0< n\leq n_T}\left|\lambda^{-1}\sum_{j=0}^{n-1}  E[F_i(t_j-h_i)|\sigma_{j-n_i}]-F_i(t_j-h_i)\right|\geq \theta_1+\theta  \right\}\leq 2\exp\left(-\frac{\theta^2}{32 T\epsilon}\right),\\
&\text{and }\P(E_4):=\P(\cup_{i=1}^M E_{4,i})\leq 2M \exp\left(-\frac{\theta^2}{32T\epsilon}\right).
\end{align*}}
\end{cor}
Note that the eligibility of equation (\ref{kchoice_ini_eq3}) as an assumption for the initial conditions follows the same idea as shown in the proof of Lemma \ref{kchoice_dn}.

\begin{cor}\label{kchoice_df2_cor}
By the proof of Lemma \ref{kchoice_df2}, on the event $E_3^c$,
\begin{align*}
\lambda^{-1} \sum_{j=0}^{n-1}\left|\left(
2 p_i\lambda\int_{t_j}^{t_{j+1}}\frac{f(t-h_i)}{l(t-h_i)}dt
\right)-E[F_i(t_j-h_i)|\sigma_{j-n_i}]\right|\leq 16\epsilon T+\sum_{j=0}^{n-1} \frac{2p_i\epsilon g(t_{j+1})}{m}.
\end{align*}
\end{cor}

The third part in the RHS of equation (\ref{kchoice_decompose_L}) is
\begin{align*}
\lambda^{-1}\sum_{j=0}^{n-1}&\left[
2\lambda\sum_{r<i}\int_{t_j}^{t_{j+1}}\int_{h_r}^{h_i}p_r\frac{w_i(t-h_i,u)}{l(t-h_i)}dudt-
\sum_{r<i}\sum_{u=h_r+\epsilon}^{h_i-\epsilon}J_{i,r}(t_j-h_i,u)
\right]\\
=\sum_{r<i}\lambda^{-1}\sum_{j=0}^{n-1}&\left[
2\lambda\int_{t_j}^{t_{j+1}}\int_{h_r}^{h_i}p_r\frac{w_i(t-h_i,u)}{l(t-h_i)}dudt-\sum_{u=h_r+\epsilon}^{h_i-\epsilon}E[J_{i,r}(t_j-h_i,u)|\sigma_{j-n_i}]\right.\\
&\left.+\sum_{u=h_r+\epsilon}^{h_i-\epsilon}E[J_{i,r}(t_j-h_i,u)|\sigma_{j-n_i}]
-\sum_{u=h_r+\epsilon}^{h_i-\epsilon}J_{i,r}(t_j-h_i,u)
\right],
\end{align*}
which will be taken care of by Lemmas \ref{k_dw1} and \ref{k_dw2} with $h(u)=n_i$.

\begin{lem}\label{k_dw1} %{tempfour}
Given any pair $r<i$. Let $h(u)$ be either $n_i:=h_i/\epsilon$ (that means the function is constant with respect to $u$) or $u\epsilon^{-1}$.
If for the initial condition we have 
\begin{align}
\sup_{0<n\leq n_T}\left|\lambda^{-1}\sum_{j=0}^{n-1}
\sum_{u=h_r+\epsilon}^{h_i-\epsilon}\left(E[J_{i,r}(t_{j-h(u)},u)|\sigma_{j-h(u)}]
-J_{i,r}(t_{j-h(u)},u)
\right)I_{\{j-h(u)\leq 0\}}
\right|
<\theta_1 \label{kchoice_ini_eq4},
\end{align}
and let 
$E_{5,i,r}(h(u)\epsilon):=$
{\footnotesize
\begin{align*}
\Bigg\{\sup_{0<n\leq n_T}\left|\lambda^{-1}\sum_{j=0}^{n-1}\left(
\sum_{u=h_r+\epsilon}^{h_i-\epsilon}E[J_{i,r}(t_{j-h(u)},u)|\sigma_{j-h(u)}]
-\sum_{u=h_r+\epsilon}^{h_i-\epsilon}J_{i,r}(t_{j-h(u)},u)
\right)
\right|
\geq\theta_1+\theta\Bigg\},
\end{align*} }
then $\P(E_{5,i,r}(h(u)\epsilon))\leq 2T\epsilon^{-1}\exp\left(-\frac{\theta^2}{32T\epsilon}\right)$ and hence
{\footnotesize
\begin{align*} 
\P(E_{5,1}):=\P(\cup_{r<i}E_{5,i,r}(h_i))
\leq M^2/2\cdot2T\epsilon^{-1}\exp\left(-\frac{\theta^2}{32T\epsilon}\right),\\
\P(E_{5,2}):=\P(\cup_{r<i}E_{5,i,r}(u))
\leq M^2/2\cdot2T\epsilon^{-1}\exp\left(-\frac{\theta^2}{32T\epsilon}\right),\\
\P(E_5):=\P(E_{5,1}\cup E_{5,2})\leq M^2T\epsilon^{-1}\exp\left(-\frac{\theta^2}{32T\epsilon}\right).
\end{align*} }
\end{lem}

\begin{lem}\label{k_dw2} %{tempfive}
Using the same setup for $h(u)$ from Lemma \ref{k_dw1}, we get
\begin{align*}
&\left|
\lambda^{-1}\sum_{j=0}^{n-1}\left[
2\lambda\int_{t_j}^{t_{j+1}}\int_{h_r}^{h_i}p_r\frac{w_i(t-h(u)\epsilon,u)}{l(t-h(u)\epsilon)}dudt-\sum_{u=h_r+\epsilon}^{h_i-\epsilon}E[J_{i,r}(t_{j-h(u)},u)|\sigma_{j-h(u)}]\right]
\right|\\
&\leq16\epsilon T+2Tp_r\epsilon(\xi+(h_i-h_r)^2\gamma)+6p_r(h_i-h_r)\epsilon \sum_{j=0}^{n-1}\frac{g(t_{j+1})}{m}.
\end{align*}
\end{lem}

For the fourth part of equation (\ref{kchoice_decompose_L}), we apply Lemmas \ref{k_dw1} and \ref{k_dw2} with $h(u)=u\epsilon^{-1}$. 
Lemmas \ref{kchoice_dn}, \ref{k_dw1}, \ref{k_dw2} and Corollaries \ref{kchoice_df1_cor}, \ref{kchoice_df2_cor} cover all terms in equation (\ref{kchoice_decompose_L}). Putting them together, we get that, on the event $(E_{4}\cup E_{5})^c$,
\begin{align*}
&\left|
\lambda^{-1}L(t_n)-l(t_n)
\right|
\leq (M^2+2M)(\theta_1+\theta)+16M\epsilon T+\sum_{j=0}^{n-1} \frac{2\epsilon g(t_{j+1})}{m}\\
\leq& (M^2+2M)(\theta_1+\theta)+16(M^2+M)\epsilon T+4 M T\epsilon(\xi+(h_M - h_1)^2\gamma)\\
&+\left[12M(h_M-h_1)\epsilon+2\epsilon\right]\sum_{j=0}^{n-1} \frac{ g(t_{j+1})}{m}.
\end{align*} 
We can extend this inequality to $t\in[t_n,t_{n+1})$ by setting $L(t):=L(t_n)$ in this interval, yielding
{\small
\begin{align} 
&\left|\frac{L(t)}{\lambda}-l(t)\right|\leq \left|\frac{L(t_n)}{\lambda}-l(t_n)\right|+\left|\int_{t_n}^t\frac{dl}{ds}ds\right|\leq \left|\frac{L(t_n)}{\lambda}-l(t_n)\right|+\epsilon(1+2+4M(h_M-h_1)\xi)\nonumber\\
&\leq 2M^2(\theta_1+\theta)+38M^2\epsilon T+4 M T\epsilon((h_M+1)\xi+(h_M)^2\gamma)+\left[12 M h_M  + 2\right]\epsilon\sum_{j=0}^{n-1} \frac{ g(t_{j+1})}{m}\nonumber\\
&\leq 2M^2(\theta_1+\theta)+38M^2 T \epsilon +8 M T (h_M)^2(\xi+\gamma)\epsilon+13 M h_M  \epsilon\sum_{j=0}^{n-1} \frac{ g(t_{j+1})}{m}\label{k_el},
\end{align}}
where for simplicity of expression we assume $T\geq 1$ and $h_M>2$.

\subsection{Difference in number of pending tips}\label{k_df_pendingtips}
In this section, we will establish equation (\ref{k_ew}) as an upper bound of $|{W_i(\cdot)}/N-w_i(\cdot)|$. As usual, we first rewrite the difference using a telescoping sum using equations (\ref{kchoice_evo_Wi1}), (\ref{kchoice_evo_Wi2}), (\ref{k_fl_wi1}), (\ref{k_fl_wi2}), (\ref{k_fl_w}):
{\footnotesize
\begin{align}
 &\left[
\frac{W_i(t_n,h_i-k\epsilon)}{N}-w_i(t_n,h_i-k\epsilon)
\right]\nonumber\\
&=  N^{-1} \{
W_i(t_n,h_i-k\epsilon)
-W_i(t_n-(k-1)\epsilon,h_i-\epsilon)- N \left[ 
w_i(t_n,h_i-k\epsilon)
- w_i(t_n-(k-1)\epsilon,h_i-\epsilon)
\right]\}\nonumber\\
&+N^{-1}
\left[
W_i(t_n-(k-1)\epsilon,h_i-\epsilon)-
N w_i(t_n-(k-1)\epsilon,h_i-\epsilon)
\right]\nonumber\\
&=    N^{-1} \sum_{j=1}^{k-1} 
\Big\{
W_i(t_{n-k+j+1},h_i- (j+1) \epsilon)-W_i(t_{n-k+j},h_i-j \epsilon)\nonumber\\
&-N\left[w_i(t_{n-k+j+1},h_i- (j+1) \epsilon)-w_i(t_{n-k+j},h_i-j \epsilon)\right]
\Big\}\nonumber\\
&+   N^{-1}\left(
F_i(t_{n-k})+\sum_{j>i} \sum_{u=h_i}^{h_j-\epsilon}J_{j,i}(t_{n-k},u)-N\left(
2p_i\frac{f(t_{n-k})}{l(t_{n-k})}+\sum_{j>i} \int_{h_i}^{h_j}2p_i\frac{w_j(t_{n-k},u)}{l(t_{n-k})}du
\right)\right.\nonumber\\
&\left.+2N \int_{0}^{\epsilon} \sum_{h_j\leq u} p_j\frac{w_i(t_{n-k}+u,h_i-u)}{l(t_{n-k}+u)}du
\right)\nonumber\\
&=   N^{-1} \sum_{j=1}^{k-1} \sum_{h_r\leq h_i-j\epsilon}
\left\{
-J_{i,r}(t_{n-k+j},h_i-j\epsilon)+
N \int_{0}^{\epsilon}2 p_r\frac{w_i(t_{n-k+j}+u,h_i-j\epsilon-u)}{l(t_{n-k+j}+u)}du
\right\}\nonumber\\
&+   N^{-1}\left(
F_i(t_{n-k})-2p_i N\frac{f(t_{n-k})}{l(t_{n-k})}
\right)+ N^{-1}\left(2N\int_{0}^{\epsilon} \sum_{h_j\leq u} p_j\frac{w_i(t_{n-k}+u,h_i-u)}{l(t_{n-k}+u)}du
\right)\nonumber\\
&+N^{-1}\left(\sum_{j>i} \sum_{u=h_i}^{h_j-\epsilon}J_{j,i}(t_{n-k},u)-N
\sum_{j>i} \int_{h_i}^{h_j}2p_i\frac{w_j(t_{n-k},u)}{l(t_{n-k})}du
\right)\label{kchoice_decompose_Wi}.
\end{align}}

In the last equality of equation (\ref{kchoice_decompose_Wi}), we have separated the difference into four terms where each of them is scaled by $N$. Using the separation technique to add and subtract conditional expectations, we will establish Lemmas \ref{k_dG1} and \ref{k_dG2} for the first scaled term and establish Lemmas \ref{k_dFi1} and \ref{k_dFi2} for the second scaled term. The fourth term will be covered by Lemma \ref{k_tempthirteen} which is derived from Lemmas \ref{k_dJ1} and \ref{k_extendw}.
Lastly, we will establish an upper bound for the third scaled term in equation (\ref{k_ew_fourthterm}). Combining these results leads to the upper bound for equation (\ref{kchoice_decompose_Wi}) as given in equation (\ref{inequality:pending_tips_semi_conclusion}). With equation (\ref{inequality:pending_tips_semi_conclusion}), we are then able to derive an upper bound on the difference between the number of pending tips and its fluid limit as given in equation (\ref{k_ew}). Let 
\begin{align*}
 G_{i,n,k,j}:=\sum_{h_r\leq h_i-j\epsilon}
\Big\{
-J_{i,r}(t_{n-k+j},h_i-j\epsilon)+E[J_{i,r}(t_{n-k+j},h_i-j\epsilon)|\sigma_{n-k+j}]
\Big\}.
\end{align*}

\begin{lem}\label{k_dG1} %{tempsix}\label{error1}
If for the initial condition we have 
\begin{align} 
\sup_{i\leq M}\sup_{0<n\leq n_T}\sup_{1\leq k\leq n_i}\left|N^{-1}\sum_{j=1}^{k-1} G_{i,n,k,j}1_{\{n-k+j< 0\}}\right|\leq\theta_1,\label{kchoice_ini_eq5}\\
\text{then }\P(E_6):=\P\left(\sup_{i\leq M}\sup_{0<n\leq n_T}\sup_{1\leq k\leq n_i}\left|
  N^{-1} \sum_{j=1}^{k-1}  G_{i,n,k,j}
\right|\geq \theta_1+\theta
\right)\nonumber\\
\leq M T h_M^2 \frac{\epsilon^{-3}}{N}
\left(
512\frac{h_M^2}{\theta^2}+\exp\left(
-\frac{\theta^2 \epsilon^2 N}{8h^2_M}\right)
\right).\nonumber
\end{align} 
\end{lem}

\begin{lem}\label{k_dG2}%{restatable}{lemma}{tempseven}
{\footnotesize
\begin{align*} 
\left|
  N^{-1} \sum_{j=1}^{k-1} \sum_{h_r\leq h_i-j\epsilon}
\left\{
-E[J_{i,r}(t_{n-k+j},h_i-j\epsilon)|\sigma_{n-k+j}]+
N \int_{0}^{\epsilon}2 p_r\frac{w_i(t_{n-k+j}+u,h_i-j\epsilon-u)}{l(t_{n-k+j}+u)}du
\right\}
\right|\\
\leq16 T M \epsilon+2p_r N \epsilon^2 \gamma +6\epsilon\sum_{j=1}^{n-1} \frac{g(t_{j})}{m}.
\end{align*}}
\end{lem}

\begin{lem}\label{k_dFi1}%{restatable}{lemma}{tempeight}
If for the initial condition we have
\begin{align}
\sup_{i\leq M} \sup_{n\in[-h_M/\epsilon,0]} \left| N^{-1} \left(F_i(t_{n})-E[F_i(t_{n})|\sigma_{n}]\right)\right|\leq \theta_1,\label{kchoice_ini_eq7}
\end{align}
then using the same proof as in Lemma \ref{k_dG1} we get
\begin{align*} 
\P(F_i(t_n)-E[F_i(t_n)|\sigma_{n}]> N\theta)\leq 256\epsilon\frac{1}{N\theta^2}+\exp{\left(-\frac{\theta^2 N}{8}\right)} \quad\text{ for }n> 0,
\end{align*} 
and hence 
\begin{align*} 
\P(E_7):=
\P(\sup_{i\leq M} \sup_{n\leq n_T} \sup_{k\leq h_M\epsilon^{-1}} N^{-1} \left(F_i(t_{n-k})-E[F_i(t_{n-k})|\sigma_{n-k}]\right)> \theta+\theta_1)\\
\leq 256\frac{1}{N\theta^2}+\exp{\left(-\frac{\theta^2 N}{8}\right)}.
\end{align*} 
\end{lem}

\begin{lem}\label{k_dFi2}%{restatable}{lemma}{tempnine}
By the proof of Lemma \ref{kchoice_df2}, we have
{\footnotesize
\begin{align*} 
N^{-1}\left|
E[F_i(t_{n-k})|\sigma_{n-k}]-2p_i\frac{f(t_n-k)}{l(t_{n-k})}
\right|&\leq& 2p_i  N^{-1} \frac{|f(t_{n-k})-F_\lambda(t_{n-k})|+  |(t_{n-k})-L_\lambda(t_{n-k})|}{m}.
\end{align*} }
\end{lem}

We now establish Lemmas \ref{k_dJ1}, \ref{k_extendw} and \ref{k_tempthirteen} to cover the fourth scaled term in the last equality of equation (\ref{kchoice_decompose_Wi}). The fourth scaled term is
\begin{align}
 &N^{-1}\left|
\sum_{j>i} \sum_{u=h_i}^{h_j-\epsilon}J_{j,i}(t_{n-k},u)-N
\sum_{j>i} \int_{h_i}^{h_j}2p_i\frac{w_j(t_{n-k},u)}{l(t_{n-k})}du
\right|\nonumber\\
\leq& N^{-1}\left|
\sum_{j>i} \sum_{u=h_i}^{h_j-\epsilon}\left(
 J_{j,i}(t_{n-k},u)-E[J_{j,i}(t_{n-k},u)|\sigma_{n-k}]
 \right)\right|\nonumber\\
 &+N^{-1}\left|\sum_{j>i} 
 \sum_{u=h_i}^{h_j-\epsilon} E[J_{j,i}(t_{n-k},u)|\sigma_{n-k}]-N
\sum_{j>i} \int_{h_i}^{h_j}2p_i\frac{w_j(t_{n-k},u)}{l(t_{n-k})}du
\right|.\label{k_lastterm_induction_eq0}
\end{align}

\begin{lem}\label{k_dJ1}%{restatable}{lemma}{tempeleven}
If for the initial condition we have 
\begin{align}
\sup_{n\in[-h_M/\epsilon,0]}\left|N^{-1}\sum_{j>i}\sum_{u=h_i}^{h_j-\epsilon}
 \left(
 J_{j,i}(t_{n},u)-E[J_{j,i}(t_{n},u)|\sigma_{n}]
 \right)\right|
 \leq \theta_1\label{kchoice_ini_eq6},
\end{align}
then $\P(E_8):=$
\begin{align*} 
&\P(\sup_{i}\sup_{n-k\leq n_T}N^{-1}|\sum_{j>i}\sum_{u=h_i}^{h_j-\epsilon}
 \left(
 J_{j,i}(t_{n-k},u)-E[J_{j,i}(t_{n-k},u)|\sigma_{n-k}]
 \right)|>\theta+\theta_1)\\
&\leq 512 M^2 h_M\frac{1}{N\theta^2}+\exp\left(
-\frac{\theta^2N}{8}\right).
\end{align*} 
\end{lem}

Next we establish Lemma \ref{k_extendw} which will be used in the proof of Lemma \ref{k_tempthirteen} and in the argument that extends equation (\ref{kchoice_temp_W}) to equation (\ref{kchoice_extend_W}).
\begin{lem}\label{k_extendw}%{restatable}{lemma}{temptwelve}\label{extendw}
Consider $s_1,s_2\in[k\epsilon,(k+1)\epsilon)$ and $s_2>s_1$. Then for $i=1,2,...,M$, $\left|w_i(t,s_2)-w_i(t,s_1)\right|\leq \alpha_{0,1}M(h_M)^M\exp(2Mh_M/m)\epsilon$, where 
\begin{align*} 
\alpha_{0,1}=(2+2Mh_M)\xi+6/m+
+\frac{4(3+4M h_M \xi)}{m}
+2h_M\xi \frac{2(3+4 M h_M \xi)}{m}.
\end{align*} 
\end{lem}
Applying Lemmas \ref{k_dJ1} and \ref{k_extendw} on equation (\ref{k_lastterm_induction_eq0}) we get the following Lemma \ref{k_tempthirteen} as an upper bound of equation (\ref{k_lastterm_induction_eq0}).

\begin{lem}\label{k_tempthirteen}%{restatable}{lemma}{tempthirteen}
In the event of $E_8^c$,
\begin{align*} 
&  N^{-1}\left|
\sum_{j>i} \sum_{u=h_i}^{h_j-\epsilon}J_{j,i}(t_{n-k},u)-N
\sum_{j>i} \int_{h_i}^{h_j}2p_i\frac{w_j(t_{n-k},u)}{l(t_{n-k})}du
\right|\\
\leq&  \theta+\theta_1
+32Mh_M\epsilon^{2}N^{-1}+2\alpha_{0,1}M^2(h_M)^{M+1}\exp(2Mh_M/m)\epsilon\\
&+2p_i\epsilon  \sum_{j>i} \sum_{u=h_i}^{h_j-\epsilon} 
\left(
\frac{ |W_j(t_{n-k},u)/N- w_j(t_{n-k},u)|}{m}
\right)
+\frac{4p_i}{m}Mh_M|L_{\lambda}(t_{n-k})-l(t_{n-k})|.
\end{align*}
\end{lem}

For the third term in last equality of equation (\ref{kchoice_decompose_Wi}), recall that $|w_i(t,u)/l(t)|<\xi$ as stated in the conditions of Theorem \ref{kchoice_main_theorem}, we have
\begin{align}
& \left|
2\int_{0}^{\epsilon} \sum_{h_j\leq u} p_j\frac{w_i(t_{n-k}+u,h_i-u)}{l(t_{n-k}+u)}du
\right|\leq 
2 \int_{0}^{\epsilon} \sum_{h_j\leq u} p_j \left|
\frac{w_i(t_{n-k}+u,h_i-u)}{l(t_{n-k}+u)}
\right|du\nonumber\\
&\leq 2\int_{0}^{\epsilon} \sum_{h_j\leq u} p_j \xi du\leq 2\epsilon \xi \label{k_ew_fourthterm}.
\end{align} 

By applying Lemmas \ref{k_dG1}, \ref{k_dG2}, \ref{k_dFi1}, \ref{k_dFi2} and \ref{k_tempthirteen} as well as equation (\ref{k_ew_fourthterm}) to equation (\ref{kchoice_decompose_Wi}), we conclude that there exist a constant $\alpha_1$ independent of $\lambda,\epsilon$ such that
\begin{align} 
&\left|
\frac{W_i(t_n,h_i-k\epsilon)}{N}-w_i(t_n,h_i-k\epsilon)
\right|\nonumber\\
&\leq 3\theta_1+3\theta +16 T M \epsilon+2p_r N \epsilon^2 \gamma +2\epsilon \xi+ +32Mh_M\epsilon^{2}N^{-1}\nonumber\\
&+2\alpha_{0,1}M^2(h_M)^{M+1}\exp(2Mh_M/m)\epsilon +6\epsilon\sum_{j=1}^{n-1} \frac{g(t_{j})}{m}\nonumber\\
&+2p_i  N^{-1} \frac{|f(t_{n-k})-F_\lambda(t_{n-k})|+  |l(t_{n-k})-L_\lambda(t_{n-k})|}{m}
\nonumber\\
&+\frac{4p_i}{m}Mh_M|L_{\lambda}(t_{n-k})-l(t_{n-k})|+2p_i\epsilon  \sum_{j>i} \sum_{u=h_i}^{h_j-\epsilon} 
\left(
\frac{ |W_j(t_{n-k},u)/N- w_j(t_{n-k},u)|}{m}
\right)\nonumber\\
&\leq3\theta_1+3\theta+\alpha_1\epsilon+6\epsilon\sum_{j=1}^{n-1} \frac{g(t_{j})}{m} 
+2  N^{-1} \frac{|f(t_{n-k})-F_\lambda(t_{n-k})|+  |l(t_{n-k})-L_\lambda(t_{n-k})|}{m}\nonumber\\
&+\frac{4}{m}Mh_M|L_{\lambda}(t_{n-k})-l(t_{n-k})|
+2p_i\epsilon  \sum_{j>i} \sum_{u=h_i}^{h_j-\epsilon} 
\left(
\frac{ |W_j(t_{n-k},u)/N- w_j(t_{n-k},u)|}{m}
\right)\nonumber\\
&=:y(n,k)+2p_i\epsilon  \sum_{j>i} \sum_{u=h_i}^{h_j-\epsilon} 
\left(
\frac{ |W_j(t_{n-k},u)/N- w_j(t_{n-k},u)|}{m}
\right) \label{inequality:pending_tips_semi_conclusion},
\end{align} 
where
\begin{align} 
y(n,k):=3\theta_1+3\theta+\alpha_1\epsilon+6\epsilon\sum_{j=1}^{n-1} \frac{g(t_{j})}{m} +\frac{4}{m}Mh_M|L_{\lambda}(t_{n-k})-l(t_{n-k})|\nonumber\\
+2   \frac{|f(t_{n-k})-F_\lambda(t_{n-k})|+  |l(t_{n-k})-L_\lambda(t_{n-k})|}{mN}.\label{y_n_and_k}
\end{align}
Equation (\ref{inequality:pending_tips_semi_conclusion}) will be bounded above using induction as follows: 
\underline{\textbf{Base case : $i=M$.}} In such case there is nothing in the term $\sum_{j>i}$
and hence 

\begin{align*} 
 \left|
\frac{W_i(t_n,h_i-k\epsilon)}{N}-w_i(t_n,h_i-k\epsilon)
\right|\leq y(n,k).
\end{align*} 
\noindent\underline{\textbf{Inductive case : $i<M$.}} First consider $i=M-1$ and use the base case above to get
{\footnotesize
\begin{align*} 
 \left|
\frac{W_i(t_n,h_i-k\epsilon)}{N}-w_i(t_n,h_i-k\epsilon)
\right|&\leq y(n,k)+2p_i\epsilon  \sum_{u=h_i}^{h_M-\epsilon} 
\left(
\frac{ |W_M(t_{n-k},u)/N- w_M(t_{n-k},u)|}{m}
\right)\\
&\leq (2h_M/m+1)y(n,k).
\end{align*}} 
Iterate on $i\leq M$ we have that there exist an $\alpha_2$ independent of $\lambda,\epsilon$ s.t for all $i$,
\begin{align} 
 \left|
\frac{W_i(t_n,h_i-k\epsilon)}{N}-w_i(t_n,h_i-k\epsilon)
\right|\leq  \alpha_2 y(n,k). \label{kchoice_temp_W}
\end{align} 

Now we extend the inequality to $t\in [t_n,t_{n+1})$ and $s\in((k-1)\epsilon,k\epsilon]$,
{\footnotesize
\begin{align} 
&\left|\frac{W_i(t,s)}{N}-w_i(t,s) \right|\nonumber\\
\leq& \left|\frac{W_2(t_n,k\epsilon)}{N}-w_2(t_n,k\epsilon) \right|+\left|w_2(t,s)-w_2(t_n,k\epsilon)\right|\nonumber\\
\leq& \left|\frac{W_2(t_n,k\epsilon)}{N}-w_2(t_n,k\epsilon) \right|+\left|w_2(t,k\epsilon-(t-t_n))-w_2(t_n,k\epsilon)\right|+\left|w_2(t,k\epsilon-(t-t_n))-w_2(t,s)\right|\nonumber\\
\leq&\left|\frac{W_2(t_n,k\epsilon)}{N}-w_2(t_n,k\epsilon) \right|+\int_{0}^{t-t_n}\left|-2\sum_{h_j\leq u}p_j\frac{w_i(t_n+u,k\epsilon-u)}{l(t_n+u)}\right|du\nonumber\\
&+\alpha_{0,1}M(h_M)^M\exp(2Mh_M/m)\epsilon \nonumber\\
\leq& \alpha_2 y(n,k)+2\xi\epsilon+\alpha_{0,1}M(h_M)^M\exp(2Mh_M/m)\epsilon\label{kchoice_extend_W},
\end{align} }
where $\alpha_{0,1}$ is defined in Lemma \ref{k_extendw}. 

Recall equation (\ref{y_n_and_k}), then there exist $\alpha_3$ independent of $\lambda, \epsilon$ such that, the event $(E_6\cup E_7\cup E_8)^c$,
\begin{align}
&\left|\frac{W_i(t,s)}{N}-w_i(t,s) \right|\label{k_ew}\\
&\leq \alpha_3\left(
\theta_1+\theta+\epsilon+\epsilon\sum_{j=1}^{n-1} \frac{g(t_{j})}{m} 
+ |f(t_{n-k})-F_\lambda(t_{n-k})| +|L_{\lambda}(t_{n-k})-l(t_{n-k})|
\right)\nonumber.
\end{align}

\subsection{Applying Gronwall's inequality}\label{kchoice_finalstep}
To summarize, we have equations (\ref{k_ef}), (\ref{k_el}) and (\ref{k_ew}) hold on the event $(\cup_{i=1}^8 E_i)^c$. Recall that 
\begin{align*} 
g(t)=\sup_{0<s\leq t}\left\{\left|\frac{F(s)}{\lambda}-f(s)\right|+\left|\frac{L(s)}{\lambda}-l(s)\right|+\sum_{i=1}^M \sup_{u\leq h_i} \left|\frac{W_i(s,u)}{N}-w_i(s,u) \right| \right\}.
\end{align*}

Since $g$ is non-decreasing, equations (\ref{k_ef}), (\ref{k_el}) and (\ref{k_ew}) imply
\begin{align*} 
&g(t_n)\leq \sup_{t\leq t_n} |F_\lambda(t)-f(t)|+\sup_{t\leq t_n} |L_\lambda(t)-l(t)|+M\sup_{t\leq t_n}\sup_{u\leq h_i}  |\frac{W_i(t,u)}{N}-w_i(s,u)|\\
&\leq (2M+2M^2+\alpha_3)(\theta_1+\theta)+(16MT+38M^2 T+8 M T (h_M)^2(\xi+\gamma)+\alpha_3)\epsilon\\
&+(2+13 M h_M+\alpha_3)\epsilon \sum_{j=0}^{n-1} \frac{ g(t_{j+1})}{m}+\alpha_3\left(|f(t_{n-1})-F_\lambda(t_{n-1})+|L_{\lambda}(t_{n-1})-l(t_{n-1})|
\right)\\
&\leq [2M+2M^2+\alpha_3(1+2M+2M^2)](\theta_1+\theta)\\
&+[16MT+38M^2 T+8 M T h_M^2(\xi+\gamma)+\alpha_3(1+54M^2 T+8 M T h_M^2(\xi+\gamma))]\epsilon\\
&+[2+13 M h_M+\alpha_3(3+13 M h_M)]\epsilon \sum_{j=0}^{n-1} \frac{ g(t_{j})}{m}+(2+13 M h_M+\alpha_3)\epsilon  \frac{ g(t_n)}{m}.
\end{align*} 
Hence there exist a constant $\alpha$ independent of $\lambda,\epsilon$ such that
\begin{align*} 
&g(t_n)\leq \frac{1}{2}\alpha\left(
\theta_1+\theta+\epsilon+\epsilon \sum_{j=0}^{n-1}  g(t_{j})
\right)
+(2+13 M h_M+\alpha_3)\epsilon  \frac{ g(t_n)}{m}.
\end{align*} 

If we assume $\epsilon$ small enough such that $(2+13 M h_M+\alpha_3)\epsilon <1/2$, then we get
$g(t_n)\leq \alpha\left(
\theta_1+\theta+\epsilon+\epsilon \sum_{j=0}^{n-1}  g(t_{j})
\right)$. By Gronwall's lemma,
\begin{align*}
g(t_n) &\leq \alpha (\theta_1+\theta+\epsilon)\exp{ 
\left(
\sum_{j=0}^{n-1} \alpha \epsilon 
\right)
}\leq \alpha (\theta_1+\theta+\epsilon)e^{\alpha T}\\
g(T)&\leq \alpha (\theta_1+\theta+\epsilon)e^{\alpha T}.
\end{align*}
Given $\delta>0$, choose $\theta=\theta_1=\delta e^{-\alpha T}/(3\alpha)$ and $\epsilon_0=\theta$. Then for $\epsilon\leq \epsilon_0$, we have $\alpha (\theta_1+\theta+\epsilon)e^{\alpha T}\leq \delta$. Therefore on the event $(\cup_{i=1}^8 E_i)^c$, $g(T)\leq \delta$. Hence
{\footnotesize
\begin{align*} 
&\P(g(T)>\delta)\leq \P(\cup_{i=1}^8 E_i)) \leq \sum_{i=1}^8 \P(E_i)\\
\leq& 2M\exp\left(-\frac{\theta^2}{32T\epsilon}\right) +2M\exp\left(-\frac{\theta^2}{32T\epsilon}\right) +T\epsilon^{-1} \exp\left(-\frac{(1-p_1h_1/2)^2\lambda}{2h_1}\right)\\
& +2M \exp\left(-\frac{\theta^2}{32 T \epsilon}\right) +M^2T\epsilon^{-1}\exp\left(-\frac{\theta^2}{32T\epsilon}\right)+ 256\frac{1}{N\theta^2}+\exp{\left(-\frac{\theta^2 N}{8}\right)}\\
& +M T h_M^2 \frac{\epsilon^{-3}}{N}
\left(
512\frac{h_M^2}{\theta^2}+\exp\left(
-\frac{\theta^2 \epsilon^2 N}{8h^2_M}\right)
\right)+ 512 M^2 h_M\frac{1}{N\theta^2}+\exp\left(
-\frac{\theta^2N}{8}\right).
\end{align*}}

Combining similar terms shows that there exists positive constant $\beta$ such that
\begin{align*} 
\P(g(T)>\delta)\leq& \beta\epsilon^{-1}\exp\left(-\beta \frac{\theta^2}{T\epsilon} \right)+\beta T\epsilon^{-1}\exp\left(-\beta\lambda\right)\\
&+\beta \exp\left(-\beta\theta^2N\right)+\beta T\frac{1}{\epsilon^3N\theta^2}+\beta T\frac{1}{\epsilon^3N}\exp\left(-\beta\theta^2\epsilon^2N\right).
\end{align*} 
Hence as $\lambda,\epsilon^{-1},N,\epsilon^3N\rightarrow \infty$ the probability goes to 0.

\section{Proof of lemmas}\label{k_proofoflemmas}

\begin{proof}[\textbf{Proof of Lemma \ref{kchoice_dn}}]
In order to prove the lemma, we divide the calculation of $D_n$ into $M$ parts and derive a probability bound for each part. Recall the definition of $D_n$ which can be expressed as follows:
\begin{align*}
D_n:=\lambda^{-1}\sum_{j=0}^{n-1} \left(-N+\sum_{i=1}^M N_i(t_j-h_i)\right)=\sum_{i=1}^M\left\{\lambda^{-1}\sum_{j=0}^{n-1} \left(N_i(t_j-h_i)-p_i N\right)\right\}.
\end{align*}

Note that $N_i(t_j)$ for $t_j\leq0$ is consider as initial conditions and recall that $n_i=h_i/\epsilon$, we then split each of the $M$ parts as
\begin{align}
\sum_{j=0}^{n-1} \frac{\left(N_i(t_j-h_i)-p_i N\right)}{\lambda}=\sum_{j=-n_i}^{-1\land (n-n_i-1)} \frac{\left(N_i(t_j)-p_i N\right)}{\lambda}+\sum_{j=0}^{n-n_i-1} \frac{\left(N_i(t_j)-p_i N\right)}{\lambda},\nonumber\\
\quad\label{kp_dn_1}
\end{align}
where the first summation on the RHS of equation (\ref{kp_dn_1}) is regarded as initial condition. Since $E[N_i(t_j)]=p_iN$ for $t_j>0$ and $N_i(t_j)$ is binomial distributed, Lemma \ref{kchoice_azuma} leads to
\begin{align} 
\P\left(\sup_{n_i< n\leq n_T} \left|\lambda^{-1}\sum_{j=0}^{n-n_i-1} \left(N_i(t_j)-p_i N\right)\right|\geq \theta\right)
\leq 2\exp\left(-\frac{\theta^2}{32T\epsilon}\right)\label{kp_dn_2}.
\end{align}
 
Recall the assumption stated in equation (\ref{kchoice_ini_eq1}) which says for a given $\theta_1$,
\begin{align} 
\sup_{-n_i\leq k<0}\left|\lambda^{-1}\sum_{j=-n_i}^{k} \left(N_i(t_j)-p_iN\right)\right|<\theta_1 \text{ for all }i=1,...,M.\label{kp_dn_3}
\end{align}
With equations (\ref{kp_dn_1}), (\ref{kp_dn_2}) and (\ref{kp_dn_3}) we get
\begin{align*}
&\P\left(\sup_{n_i\leq n\leq n_T} \left|\lambda^{-1}\sum_{j=0}^{n-1} \left(N_i(t_j-h_i)-p_i N\right)\right|\geq \theta+\theta_1\right)\leq 2\exp\left(-\frac{\theta^2}{32T\epsilon}\right)\\
&\implies\P\left(\sup_{0<n\leq n_T}|D_n|\geq M\theta+M\theta_1\right)\\
&\leq \sum_{i=1}^M \P\left(\sup_{n_i\leq n\leq n_T} \left|\lambda^{-1}\sum_{j=0}^{n-1} \left(N_i(t_j-h_i)-p_i N\right)\right|\geq \theta+\theta_1\right)\leq 2M\exp\left(-\frac{\theta^2}{32T\epsilon}\right),
\end{align*}
and the lemma follows. 
\end{proof}

\begin{proof}[\textbf{Proof of Lemma \ref{lem_Initial_condition}}]
The proof follows the same idea as for equation (\ref{kp_dn_2}). By the assumption of Lemma \ref{lem_Initial_condition}, the quantities $N_i(t_j)$ for $i=1,\cdots,M$ and $j=-n_i,\cdots,-1$ in the initial condition follow the same distribution as in equation (\ref{kchoice_distribution_Ni}). Hence, applying the same derivation as for equation (\ref{kp_dn_2}) using Lemma \ref{kchoice_azuma}, we obtain
\begin{align*}
\P\left(\sup_{-n_i\leq k<0}\left|\lambda^{-1}\sum_{j=-n_i}^{k} \left(N_i(t_j)-p_iN\right)\right| \geq \theta_1 
\right) \leq 2\exp\left(-\frac{\theta_1^2}{32h_M\epsilon}\right)\ \text{for }i=1,\cdots,M.
\end{align*}
Therefore, the initial condition of interest is violated with probability
\begin{align*}
&\P\left(\bigcup_{i=1}^{M}\left\{\sup_{-n_i\leq k<0}\left|\lambda^{-1}\sum_{j=-n_i}^{k} \left(N_i(t_j)-p_iN\right)\right|\geq\theta_1\right\}\right)\\
    \leq&\sum_{i=1}^M\P\left(\sup_{-n_i\leq k<0}\left|\lambda^{-1}\sum_{j=-n_i}^{k} \left(N_i(t_j)-p_iN\right)\right| \geq \theta_1 
\right) \leq 2M\exp\left(-\frac{\theta_1^2}{32h_M\epsilon}\right),
\end{align*}
which approaches zero as $\epsilon\rightarrow 0$. This implies that when the parameters $\lambda,N,\epsilon$ approach their limit, the desired initial condition holds with probability approaching one.

In equation (\ref{kp_dn_1}) we decompose the term into a part with the initial condition and a part involving the variables generated after $t=0$, and both parts follow the same rule when generated. Therefore, it is sufficient to derive the probability bound for the latter part by assuming that the initial condition follows a similar behavior. The same idea can be applied to other assumptions on the initial condition, and hence we will omit the lemmas and proofs on the validity of these assumptions. Since only finitely many such assumptions are required, the probability that any of them is violated also vanishes and hence Proposition \ref{k_prop_ini} follows.
\end{proof}

\begin{proof}[\textbf{Proof of Lemma \ref{kchoice_E3}}]
By the construction of the process, vertices that finish its POW between the time $t_{n-h_1\epsilon^{-1}}$ and $t_{n-1}$ will still be a tip at time $t_n$, i.e.
\begin{align*}
L(t_{n})&\geq\sum_{j=n-h_1\epsilon^{-1}}^{n-1}\sum_{i=1}^{M} N_i(t_j-h_i)\\
L_{\lambda}(t_{n})&\geq\lambda^{-1}\sum_{j=n-h_1\epsilon^{-1}}^{n-1} N_1(t_j-h_1)=:\lambda^{-1}\Sigma(n,-)+\lambda^{-1}\Sigma(n,+),
\end{align*}
where $L_\lambda(t_n):=L(t_n)/\lambda$, and $\Sigma(n,-)$ includes the summands that are consider as initial condition while $\Sigma(n,+)$ includes the summands that are generated after $t=0$ respectively.

Recall that $N_1(t_k)$ is a binomial random variable which counts the number of arrivals having POW duration as $h_1$. Hence with $P(X_i=1)=p_1$ we can rewrite $\Sigma(n,+)=\sum_{j=1}^{r} X_i$ for some $r \leq c_1\lambda$ (since some $N_1$ are in initial conditions) and the value of $r$ depends on $t_n$.
By Azuma–Hoeffding inequality 
\begin{align*}
\P\left(\lambda^{-1}\sum_{j=1}^{r} (X_i-p_1)\leq -\theta\right)=\P\left(\sum_{j=1}^{r} (X_i-p_1)\leq -\theta\lambda\right)\\
\leq\exp\left(-\frac{\theta^2\lambda^2}{2r}\right)
\leq\exp\left(-\frac{\theta^2\lambda^2}{2h_1\lambda}\right)=\exp\left(-\frac{\theta^2\lambda}{2h_1}\right).
\end{align*}

\textbf{Case 1}: $r\geq \frac{1}{2}h_1\lambda$ (i.e. $n>3h_1\epsilon^{-1}/2$)
\begin{align*} 
\P({L_\lambda(t_n)\leq 1})&\leq\P({\lambda^{-1}\Sigma(n,+)\leq 1})=\P(\lambda^{-1}\sum_{j=1}^r(X_i-p_1)\leq 1-p_1r/\lambda)\\
&\leq\exp\left(-\frac{(1-p_1r/\lambda)^2\lambda}{2h_1}\right)\leq\exp\left(-\frac{(1-p_1h_1/2)^2\lambda}{2h_1}\right),
\end{align*} 
where, without loss of generality, we assume $h_1p_1/2>1$ since we can change the time unit and therefore $(1-p_1r/\lambda)^2\geq (1-p_1h_1)^2$.
Then
\begin{align*} 
\P(\inf_{h_1\epsilon^{-1}/2\leq n\leq n_T}L_\lambda(t_n)\leq 1)&\leq \sum_{n=h_1\epsilon^{-1}/2}^{n_T}\P(L_\lambda(t_n)\leq 1)\leq\sum_{n=h_1\epsilon^{-1}/2}^{n_T}\P(L_\lambda(t_n)\leq 1)\\
&\leq T\epsilon^{-1}\exp\left(-\frac{(1-p_1h_1/2)^2\lambda}{2h_1}\right).
\end{align*} 

\textbf{Case 2}:$r< \frac{1}{2}h_1\lambda$  (i.e. $n\leq 3h_1\epsilon^{-1}/2$). Focus on $\Sigma(n,-)$ with number of $X_i$ larger than $\frac{1}{2}h_1\lambda$. By the same idea as Case 1, if we consider the initial condition is randomly generated following the evolution rule, it is reasonable to assume 
\begin{align*} 
\inf_{-h_M\epsilon^{-1}\leq n\leq 3/2h_1\epsilon^{-1}}\sum_{j=n-2h_1\epsilon^{-1}}^{0} N_1(t_j)>1,
\end{align*} 
which implies
\begin{align*} 
\P\left(\inf_{-h_M\epsilon^{-1}\leq n\leq n_T}L_{\lambda}(t_n)\leq 1\right)\leq T\epsilon^{-1} \exp\left(-\frac{(1-p_1h_1/2)^2\lambda}{2h_1}\right).
\end{align*} 
\end{proof}
%\tempthree*
\begin{proof}[\textbf{Proof of Lemma \ref{kchoice_df2}}]

Recall equation (\ref{kchoice_distribution_Fi1}) and we rewrite it as
\begin{align*}
F_i(t_n) = \sum_{j=1}^{F(t_n)-F_1(t_n)-...-F_{i-1}(t_n)} X_j\,: \,\, X_j\in \{0,1\}, \,\, \P(X_j=0)=\left(1 - \frac{1}{L(t_n)}\right)^{2 N_i(t_n)}.
\end{align*}
Since $\sum_{r=1}^{i}F_r(t_n)$ are the number of free tips that are selected at time $t_n$ by any arrival with Type no larger than $i$, then we also have 
\begin{align*} 
\sum_{r=1}^{i}F_r(t_n) &=& \sum_{j=1}^{F(t_n)} X_j\,: \,\, X_j\in \{0,1\}, \,\, \P(X_j=0)=\left(1 - \frac{1}{L(t_n)}\right)^{2 \sum_{r=1}^{i}N_r(t_n)},
\end{align*} 
and hence
\begin{align} 
&E\left[\sum_{r=1}^{i}F_r(t_j)\Big|\sigma_{j}\right]
=E\left[F(t_j)\left[1-\left(1-\frac{1}{L(t_j)}\right)^{2\sum_{r=1}^i N_r(t_j)}\right]\Big|\sigma_{j}\right]\label{kchoice_df2_pf_1}\\
&=E\left[F(t_j)\left[\frac{2\sum_{r=1}^i N_r(t_j)}{L(t_j)}-\sum_{d=2}^{2\sum_{r=1}^i N_r(t_j)}{2\sum_{r=1}^i N_r(t_j) \choose d}\left(-\frac{1}{L(t_j)}\right)^d\right]\Big|\sigma_{j}\right]\nonumber\\
&=\frac{2\sum_{r=1}^i p_r N F(t_j)}{L(t_j)}-F(t_j)E\left[\sum_{d=2}^{2\sum_{r=1}^i N_r(t_j)}{2\sum_{r=1}^i N_r(t_j) \choose d}\left(-\frac{1}{L(t_j)}\right)^d\Big|\sigma_{j}\right].\nonumber
\end{align}

Let $Y$ be a random variable such that $0\leq Y\leq N$ and denote $F_\lambda(t)=F(t)/\lambda$ and $L_\lambda=L(\lambda)$
\begin{align}
&F(t_j)\left|E\left[\sum_{d=2}^{2Y}{2Y \choose d}(-\frac{1}{L(t_j)})^d|\sigma_{j}\right]\right|  
\leq F(t_j)\left|E\left[\sum_{d=2}^{2Y}{2Y \choose d}(\frac{1}{L(t_j)})^d|\sigma_{j}\right]\right|\nonumber\\
&\leq F(t_j)\left|E\left[\sum_{d=2}^{2N}{2N \choose d}(\frac{1}{L(t_j)})^d|\sigma_{j}\right]\right|
\leq F(t_j)\left|\sum_{d=2}^{2N}{2N \choose d}\frac{1}{L(t_j)^d}\right|\nonumber\\
&\leq F(t_j)\left|\sum_{d=2}^{2N}(2N)^d\frac{1}{L(t_j)^d}\right|
\leq \frac{F(t_j)}{L(t_j)}\left |\sum_{d=2}^{2N}\frac{(2N)^d}{\lambda^{d-1}}\frac{1}{L_\lambda(t_j)^{d-1}}\right|\label{kchoice_df2_pf_2},
\end{align} 
where 
\begin{align}
\frac{F(t_j)}{L(t_j)}\left |\sum_{d=2}^{2N}\frac{(2N)^d}{\lambda^{d-1}}\frac{1}{L_\lambda(t_j)^{d-1}}\right|\leq\frac{4\epsilon^2\lambda}{1-2\epsilon}
\leq 8\epsilon^2\lambda\label{kchoice_df2_pf_3},
\end{align}
because we have restricted our analysis in the event $E_3^C$ ( recall that it says $L_\lambda>1$) and we make an assumption that $\epsilon<1/4$ for simplicity. Note that we can replace $F$ by $W_{i}\leq 2N$ in equation (\ref{kchoice_df2_pf_3})to get the bound $16\epsilon^3\lambda$ which we will use later. Equations (\ref{kchoice_df2_pf_1}), (\ref{kchoice_df2_pf_2}) and (\ref{kchoice_df2_pf_3}) then lead to
\begin{align}
E[F_i(t_j)|\sigma_j]=E\left[\sum_{r=1}^{i}F_r(t_j)-\sum_{r=1}^{i-1}F_r(t_j)\Big|\sigma_{j}\right]
=\frac{2p_iNF(t_j)}{L(t_j)}+error\label{kchoice_df2_pf_4},
\end{align}
where $|error|<16\epsilon^2\lambda$. On the other hand,
\begin{align}
&\left|\frac{f(s)}{l(s)}-\frac{F(s)}{L(s)}\right|=\left|\frac{f(s)-F_\lambda(s)}{l(s)}+\frac{F_\lambda(s)}{L_\lambda(s)}\cdot\frac{L_\lambda(s)-l(s)}{l(s)}\right|\nonumber\\
\leq&\frac{|f(s)-F_\lambda(s)|+|L_\lambda(s)-l(s)|}{l(s)}\leq\frac{|f(s)-F_\lambda(s)|+|L_\lambda(s)-l(s)|}{m}\leq g(s)/m\label{kchoice_df2_pf_5},
\end{align}
where $0<m\leq l(s)$ is the lower bound assumption for $l(s)$. And for future reference, if $s\leq 0$ the above equation is just 0 by construction of initial condition of the fluid limit. 

Putting equations (\ref{def_H}), (\ref{kchoice_df2_pf_4}), (\ref{kchoice_df2_pf_5}) together we get
\begin{align*} 
&|H_{i,j+1}|=\left|\lambda\int_{t_j}^{t_{j+1}} 2p_i\frac{f(s)}{l(s)}ds-\frac{2p_iNF(t_j)}{L(t_j)}-error\right|\\
&=\left|\lambda\int_{t_j}^{t_{j+1}}\left( 2p_1\frac{f(s)}{l(s)}-2p_1\frac{F(t_j)}{L(t_j)}\right)ds-error\right|\leq \frac{2p_iNg(t_{j+1})}{m}+16\epsilon^2\lambda\\
\implies&\lambda^{-1}\sum_{j=0}^{n-1}|H_{i,j+1}|\leq  16\epsilon T+\sum_{j=0}^{n-1} \frac{2p_i\epsilon g(t_{j+1})}{m}.
\end{align*} 
\end{proof}

%Error in L
%\tempfour*
\begin{proof}[\textbf{Proof of Lemma \ref{k_dw1}}]
First consider the following with $n$ and $i$ fixed
\begin{align*} 
&\lambda^{-1}\sum_{j=0}^{n-1}\left(
\sum_{u=h_r+\epsilon}^{h_i-\epsilon}E[J_{i,r}(t_{j-h(u)},u)|\sigma_{j-h(u)}]
-\sum_{u=h_r+\epsilon}^{h_i-\epsilon}J_{i,r}(t_{j-h(u)},u)
\right)\\
=:&\lambda^{-1}\sum_{j=0}^{n-1}\sum_{u=h_r+\epsilon}^{h_i-\epsilon}\left(\Delta(j,u)\right),
\end{align*}
where $\Delta(j,u):=E[J_{i,r}(t_{j-h(u)},u)|\sigma_{j-h(u)}]
-J_{i,r}(t_{j-h(u)},u)$. Despite the potential terms that are initial condition which are bounded by $\theta_1$, the rest are sum of random variables with expected value of $0$. Each term $\Delta(j,u)$ is measurable by $\sigma_{j-h(u)+1}$ but not by $\sigma_{j-h(u)}$, then we can group the terms $\Delta(j,u)$ by the first time they are measurable. Then each grouped sum is a $Y_s$ in Lemma \ref{kchoice_azuma} since the total number of jumps possible at each time is $2N$, then these are just a sum of $Y_s$ with $s$ at most $n_T:=T/\epsilon$. Then applying Lemma \ref{kchoice_azuma}
\begin{align*} 
&\P\left(
\left|
\lambda^{-1}\sum_{j=0}^{n-1}\left(
\sum_{u=h_r+\epsilon}^{h_i-\epsilon}E[J_{i,r}(t_{j-h(u)},u)|\sigma_{j-h(u)}]
-\sum_{u=h_r+\epsilon}^{h_i-\epsilon}J_{i,r}(t_{j-h(u)},u)
\right)
\right|
\geq \theta_1+\theta
\right)\\
&\leq 2\exp\left(-\frac{\theta^2}{32\epsilon}\right).
\end{align*} 
Then using union bound we get 
{\small
\begin{align*} 
&\P\left(
\sup_{0<n\leq n_T}\left|\lambda^{-1}\sum_{j=0}^{n-1}\left(
\sum_{u=h_r+\epsilon}^{h_i-\epsilon}E[J_{i,r}(t_{j-h(u)},u)|\sigma_{j-h(u)}]
-\sum_{u=h_r+\epsilon}^{h_i-\epsilon}J_{i,r}(t_{j-h(u)},u)
\right)
\right|
\geq\theta_1+\theta
\right)\\
&\leq 2T\epsilon^{-1}\exp\left(-\frac{\theta^2}{32T\epsilon}\right).
\end{align*}}
\end{proof}

%\tempfive*
\begin{proof}[\textbf{Proof of Lemma \ref{k_dw2}}]
Following the same idea as used in the proof of Lemma \ref{kchoice_df2},
{\footnotesize
\begin{align*} 
E\left[\sum_{r\leq k}J_{i,r}(t_{j-h(u)},u)|\sigma_{j-h(u)}\right]&= W_{i}(t_{j-h(u)},u)E\left[ 
1-\left(
1-\frac{1}{L(t_{j-h(u)})}
\right)^{2\sum_{r\leq k}N_r}|\sigma_{j-h(u)}
\right]\\
&=\frac{2\sum_{r\leq k} p_r N W_{i}(t_{j-h(u)},u)}{L(t_{j-h(u)})}+error,
\end{align*}}
for $h_k\leq u$, where $|error|\leq 16\epsilon^3\lambda$. Hence 
\begin{align*} 
&E[J_{i,r}(t_{j-h(u)},u)|\sigma_{j-h(u)}]\\
=&E[\sum_{r\leq k}J_{i,r}(t_{j-h(u)},u)|\sigma_{j-h(u)}]-E[\sum_{r\leq k-1}J_{i,r}(t_{j-h(u)},u)|\sigma_{j-h(u)}]\\
=&\frac{2 p_r N W_{i}(t_{j-h(u)},u)}{L(t_{j-h(u)})}+error^*,
\end{align*}
where $|error^*|\leq 32\epsilon^3\lambda$. Therefore, 
{\footnotesize
\begin{align} 
&\left|
\lambda^{-1}\sum_{j=0}^{n-1}\left[
2\lambda\int_{t_j}^{t_{j+1}}\int_{h_r}^{h_i}p_r\frac{w_i(t-h(u)\epsilon,u)}{l(t-h(u)\epsilon)}dudt-\sum_{u=h_r+\epsilon}^{h_i-\epsilon}E[J_{i,r}(t_{j-h(u)},u)|\sigma_{j-h(u)}]\right]
\right|\nonumber\\
\leq&\left|
\lambda^{-1}\sum_{j=0}^{n-1}\left[
2\lambda\int_{t_j}^{t_{j+1}}\int_{h_r}^{h_i}p_r\frac{w_i(t-h(u)\epsilon,u)}{l(t-h(u)\epsilon)}dudt-\sum_{u=h_r+\epsilon}^{h_i-\epsilon}\frac{2p_r N W_{i}(t_{j-h(u)},u)}{L(t_{j-h(u)})}-error^*
\right]
\right|\nonumber\\
\leq&\left|
\lambda^{-1}\sum_{j=0}^{n-1}\left[
2\lambda p_r\int_{t_j}^{t_{j+1}}
\left(
\int_{h_r}^{h_i}\frac{w_i(t-h(v)\epsilon,v)}{l(t-h(v)\epsilon)}dv-\sum_{u=h_r+\epsilon}^{h_i-\epsilon}\frac{ W_{i}(t-h(u)\epsilon,u)}{L(t-h(u)\epsilon)}
\right) dt
\right]
\right|\nonumber\\
&+32\epsilon T(h_i-h_r),\label{k_dw2_pf_eq1}
\end{align}}
where $h(u)\epsilon$ is either $h_i$ or $u$. As usual we will combine the integration and summation in the last equation. First, note that
\begin{align} 
\left|\frac{w_i(t,s)}{l(t)}-\frac{W_i(t,s)/N}{L(t)/\lambda}\right|&\leq\left|
\frac{w_i(t,s)-W_i(t,s)/N}{l(t)}+\frac{W_i(t,s)/N}{L_\lambda(t)}\cdot\frac{L_{\lambda}(t)-l(t)}{l(t)}
\right|\nonumber\\
&\leq\frac{g(t)}{m}+2\frac{g(t)}{m}\leq3\frac{g(t)}{m}\label{k_dw2_pf_eq2}.
\end{align} 

In the case where $h(u)\epsilon=h_i$, the following can be obtained using equation (\ref{k_dw2_pf_eq2})
\begin{align}
&\left|\int_{h_r}^{h_i}\frac{w_i(t-h_i,v)}{l(t-h_i)}dv-\sum_{u=h_r+\epsilon}^{h_i-\epsilon}\frac{ W_{i}(t-h_i,u)}{L(t-h_i)}\right|\nonumber\\
&\leq \left|
\sum_{u=h_r+\epsilon}^{h_i-\epsilon}\left(
\int_{u-\epsilon}^u \frac{w_i(t-h_i,v)}{l(t-h_i)}dv-\frac{ W_{i}(t-h_i,u)}{L(t-h_i)}
\right)
\right|
+\left| \int_{h_i-\epsilon}^{h_i}\frac{w_i(t-h_i,v)}{l(t-h_i)}dv\right|\nonumber\\
&\leq \left|
\sum_{u=h_r+\epsilon}^{h_i-\epsilon}
\int_{u-\epsilon}^u
\left(
\frac{w_i(t-h_i,v)}{l(t-h_i)}-\frac{ W_{i}(t-h_i,v)/N}{L(t-h_i)/\lambda}
\right)dv
\right|
+\xi\epsilon\nonumber\\
&\leq\xi\epsilon+\sum_{u=h_r+\epsilon}^{h_i-\epsilon}3\epsilon \frac{g(t)}{m}\leq\xi\epsilon+3(h_i-h_r)\frac{g(t)}{m}\label{k_dw2_pf_eq3},
\end{align} 
where $\xi$ is upper bound for $\frac{w_i(t,u)}{l(t)}$ for all $i,t,s$. Similarly in the case where $h(u)\epsilon=u$
{
\begin{align}
&\left|\int_{h_r}^{h_i}\frac{w_i(t-v,v)}{l(t-v)}dv-\sum_{u=h_r+\epsilon}^{h_i-\epsilon}\frac{ W_{i}(t-u,u)}{L(t-u)}\right|\nonumber\\
&\leq\left|
\int_{h_r}^{h_i}\frac{w_i(t-v,v)}{l(t-v)}dv
-\epsilon\sum_{u=h_r+\epsilon}^{h_i}\frac{w_i(t-u,u)}{l(t-u)}
\right|\nonumber\\
&+\left|
\epsilon\sum_{u=h_r+\epsilon}^{h_i-\epsilon}\frac{w_i(t-u,u)}{l(t-u)}
-\sum_{u=h_r+\epsilon}^{h_i-\epsilon}\frac{ W_{i}(t-u,u)}{L(t-u)}
\right|+\left|\epsilon\frac{w_i(t-h_i,h_i)}{l(t-h_i)}\right|\nonumber\\
&\leq\gamma(h_i-h_r)^2\epsilon+3\epsilon \sum_{u=h_r+\epsilon}^{h_i-\epsilon}\frac{g(t)}{m}+\xi\epsilon=\gamma(h_i-h_r)^2\epsilon+3(h_i-h_r)\frac{g(t)}{m}+\xi\epsilon\label{k_dw2_pf_eq4},
\end{align} }
where $\gamma$ is upper bound of $|\frac{\partial}{\partial u}\frac{w_i(t-u,u)}{l(t-u)}|$ for all $t,u$.
Then putting equations (\ref{k_dw2_pf_eq1}), (\ref{k_dw2_pf_eq3}) and (\ref{k_dw2_pf_eq4}) together we get
\begin{align*} 
&\left|
\lambda^{-1}\sum_{j=0}^{n-1}\left[
2\lambda\int_{t_j}^{t_{j+1}}\int_{h_r}^{h_i}p_r\frac{w_i(t-h(u)\epsilon,u)}{l(t-h(u)\epsilon)}dudt-\sum_{u=h_r+\epsilon}^{h_i-\epsilon}E[J_{i,r}(t_{j-h(u)},u)|\sigma_{j-h(u)}]\right]
\right|\\
&\leq32\epsilon T+2Tp_r\epsilon(\xi+(h_i-h_r)^2\gamma)+6p_r(h_i-h_r)\epsilon \sum_{j=0}^{n-1}\frac{g(t_{j+1})}{m}.
\end{align*} 
\end{proof}

%\tempsix*

\begin{proof}[\textbf{Proof of Lemma \ref{k_dG1}}]
Consider N balls randomly drop into $L$ boxes with equal probability and each ball have probability $1-p$ to disappear. Let $X$ denote the number of boxes that satisfy the following requirement: a). These boxes are the first W boxes. b). Each Boxes has at least one ball. Let $Y$ denote the number of balls that drops into the first $W$ boxes. $W$ and $L$ are constant. Then with $p=p_r$, $W$ corresponds to $W_i(t,u)$, $L$ corresponds to $L(t)$ and $X$ corresponds to $J_{i,r}$ the number of pending tips that change to Type $r$,  
We then have $X\leq Y$, $E[Y]=pNW/L$ and $E[X]=pNW/L+error$ where $|error|\leq 32\epsilon^3\lambda$ is deduced from conditional expectation of $J_{i,r}$ (see proof of Lemma \ref{k_dw2}. And similar result of bounded error can be generalize to any bounded W and will be used later.) Then
\begin{align}
-32\epsilon^3\lambda&\leq \E[X-Y]\label{kpf_dj_eq1},\\
X-Y&\leq 0\label{kpf_dj_eq2},
\end{align} 
using which, with Markov inequality, we get that for any $\delta>0$, $\P(X-Y\leq -\delta)< 32\epsilon^3\lambda\delta^{-1}$. This with Hoeffding's inequality for $Y$ leads to
\begin{align*} 
\P(|X-E[Y]|>N\theta/2)&\leq \P(X-Y< -N\theta/4)+\P(|Y-E[Y]|>N\theta/4)\\
&\leq 32\epsilon^3\lambda\frac{16}{N^2\theta^2}+\exp\left(
-\frac{\theta^2N^2}{8N}\right)\leq 512\epsilon^2\frac{1}{N\theta^2}+\exp\left(
-\frac{\theta^2N}{8}\right).
\end{align*}  

Assume $\epsilon$ small enough such that  $E[Y-X]/N< \theta/2$, we get
\begin{align} 
\P(|X-E[X]|>N\theta)&\leq \P(|X-E[Y]|>N\theta/2)\leq 512\epsilon^2\frac{1}{N\theta^2}+\exp\left(
-\frac{\theta^2N}{8}\right).\label{eq_ballinbox}
\end{align} 

Recall that
\begin{align*}
 G_{i,n,k,j}:=\sum_{h_r\leq h_i-j\epsilon}-J_{i,r}(t_{n-k+j},h_i-j\epsilon)
+\sum_{h_r\leq h_i-j\epsilon}E[J_{i,r}(t_{n-k+j},h_i-j\epsilon)|\sigma_{n-k+j}]
,
\end{align*}
where the first summation counts the number of pending tips that changes Type and the second summation counts the corresponding expected value. Let $p=\sum_{h_r\leq h_i-j\epsilon}p_r$, this implies that each arriving vertex has probability $p$ to have POW duration less or equal than $h_i-j\epsilon$ and $X$ represents the number of pending tips with RLT $h_i-j\epsilon$ that changes Type. Therefore, we can use equation (\ref{eq_ballinbox}) to get
\begin{align*} 
&\P(\left|N^{-1}G_{i,n,k,j}1_{\{n-k+j\geq 0\}}\right|\geq \epsilon\theta)\leq 512\frac{1}{N\theta^2}+\exp\left(
-\frac{\theta^2 \epsilon^2 N}{8}\right),
\end{align*} 
where $1_{\{\}}$ is the indicator function. Using the fact that $k\leq h_M\epsilon^{-1}$, we  get
{\small
\begin{align*} 
&\P\left(\left|N^{-1}\sum_{j=1}^{k-1} G_{i,n,k,j}1_{\{n-k+j\geq 0\}}\right|\geq\theta\right)\leq\P\left(\left|N^{-1}\sum_{j=1}^{k-1} G_{i,n,k,j}1_{\{n-k+j\geq 0\}}\right|\geq\frac{k-1}{h_M\epsilon^{-1}}\theta\right)\\
&\leq \sum_{j=1}^{k-1} \P(\left|N^{-1}G_{i,n,k,j}1_{\{n-k+j\geq 0\}}\right|\geq h_M^{-1}\epsilon\theta)\leq h_M\epsilon^{-1}\left(512\frac{h_M^2}{N\theta^2}+\exp\left(
-\frac{\theta^2 \epsilon^2 N}{8h^2_M}\right)
\right).
\end{align*}}

Therefore with the assumption that equation (\ref{kchoice_ini_eq5}) is true, we have
\begin{align*}
&\P\left(\left|N^{-1}\sum_{j=1}^{k-1} G_{i,n,k,j}\right|\geq\theta+\theta_1\right)\leq h_M\epsilon^{-1}\left(512\frac{h_M^2}{N\theta^2}+\exp\left(
-\frac{\theta^2 \epsilon^2 N}{8h^2_M}\right)
\right)\\
\implies&\P\left(\sup_{i\leq M}\sup_{0<n\leq n_T}\sup_{1\leq k\leq n_i}\left|
  N^{-1} \sum_{j=1}^{k-1}  G_{i,n,k,j}
\right|\geq \theta_1+\theta
\right)\\
&\leq M T h_M^2 \frac{\epsilon^{-3}}{N}
\left(
512\frac{h_M^2}{\theta^2}+\exp\left(
-\frac{\theta^2 \epsilon^2 N}{8h^2_M}\right)
\right).
\end{align*} 
\end{proof}

%\tempseven*

\begin{proof}[\textbf{Proof of Lemma \ref{k_dG2}}]
We will use statements derived in the proof for Lemma \ref{k_dw2}.
\begin{align*} 
E[J_{i,r}(t_{n-k+j},h_i-j\epsilon)|\sigma_{n-k+j}]&=
\frac{2p_r N W_{i}(t_{n-k+j},h_i-j\epsilon)}{L(t_{n-k+j})}+error,
\end{align*} 
where $|error|\leq 16\epsilon^3\lambda$. Hence 
{\footnotesize
\begin{align*} 
\left|\frac{1}{N}\sum_{j=1}^{k-1} \left(\sum_{h_r\leq h_i-j\epsilon}
-E[J_{i,r}(t_{n-k+j},h_i-j\epsilon)|\sigma_{n-k+j}]+
N \int_{0}^{\epsilon}2 p_r\frac{w_i(t_{n-k+j}+u,h_i-j\epsilon-u)}{l(t_{n-k+j}+u)}du\right)
\right|\\
\leq \left|
  N^{-1} \sum_{j=1}^{k-1} M
\left\{-
\frac{2p_r N W_{i}(t_{n-k+j},h_i-j\epsilon)}{L(t_{n-k+j})}+error+2 p_r N\epsilon \frac{w_i(t_{n-k+j},h_i-j\epsilon)}{l(t_{n-k+j})}
\right\}
\right|\quad\quad\quad\quad\quad\quad\\
\quad+\left|
\frac{1}{N}\sum_{j=1}^{k-1} M
\left\{
-2 p_r N\epsilon \frac{w_i(t_{n-k+j},h_i-j\epsilon)}{l(t_{n-k+j})}+
N \int_{0}^{\epsilon}2 p_r\frac{w_i(t_{n-k+j}+u,h_i-j\epsilon-u)}{l(t_{n-k+j}+u)}du
\right\}
\right|\quad\quad\quad\\
\leq 16 T M \epsilon+\sum_{j=1}^{k-1}6\epsilon \frac{g(t_{n-k+j})}{m}
+2p_r N \epsilon^2 \gamma\leq 16 T M \epsilon+2p_r N \epsilon^2 \gamma +6\epsilon\sum_{j=1}^{n-1} \frac{g(t_{j})}{m}.\quad\quad\quad\quad\quad\quad\quad\quad
\end{align*}}
\end{proof}

%\tempeleven*
\begin{proof}[\textbf{Proof of Lemma \ref{k_dJ1}}]
The term $\sum_{j>i} \sum_{u=h_i}^{h_j-\epsilon}J_{j,i}(t_{n-k},u)$
describes the number of pending tips, which has Type higher than $i$ and has RLT longer than $h_i$, that are selected by {$Type\ i$} arrivals. Hence by repeating the proof of Lemma \ref{k_dG1} with suitable choice of $W$ as shown in the following, we get
\begin{align*} 
W=\sum_{j>i} \sum_{u=h_i}^{h_j-\epsilon}W_{j}(t_{n-k},u)\leq M h_M\epsilon^{-1} (2N)=2Mh_M\lambda\\
\implies
E[X]=\frac{NW}{L}+error \quad\quad\quad\text{where }|error|\leq32Mh_M\epsilon^2\lambda.
\end{align*} 
Hence with the assumption of equation (\ref{kchoice_ini_eq6}),
{
\begin{align*} 
&\P\left(\left|\sum_{j>i}\sum_{u=h_i}^{h_j-\epsilon}
 \left(
 J_{j,i}(t_{n-k},u)-E[J_{j,i}(t_{n-k},u)|\sigma_{n-k}]
 \right)\right|>N(\theta+\theta_1)\right)\\
&\leq512 M h_M\epsilon\frac{1}{N\theta^2}+\exp\left(
-\frac{\theta^2N}{8}\right).
\end{align*} }
Therefore 
\begin{align*} 
&\P\left(\sup_{i}\sup_{n-k\leq n_T}N^{-1}\left|\sum_{j>i}\sum_{u=h_i}^{h_j-\epsilon}
 \left(
 J_{j,i}(t_{n-k},u)-E[J_{j,i}(t_{n-k},u)|\sigma_{n-k}]
 \right)\right|>\theta+\theta_1\right)\\
&\leq 512 M^2 h_M \frac{1}{N\theta^2}+\exp\left(
-\frac{\theta^2N}{8}\right).
\end{align*} 
\end{proof}

%\temptwelve*
\begin{proof}[Proof of Lemma \ref{k_extendw}]
It is easier for proof to replace $s_j$ by $h_i-s_j$, that is, we focus on
\begin{align*} 
w_i(t,h_i-s_1)&= w_i\big(t-s_1,h_i\big)+\int_{0}^{s_1} \sum_{h_j\leq h_i-u} -2p_j\frac{w_i\big(t-s_1+u,h_i-u\big)}{l(t-s_1+u)}du,
\end{align*} 
and
{
\begin{align*} 
&w_i(t,h_i-s_2)= w_i\big(t-s_2,h_i\big)+\int_{0}^{s_2} \sum_{h_j\leq h_i-u} -2p_j\frac{w_i\big(t-s_2+u,h_i-u\big)}{l\big(t-s_2+u\big)}du\\
&=w_i(t-s_2,h_i)+\int_{s_1-s_2}^{s_1} \sum_{h_j\leq h_i-(u-s_1+s_2)} -2p_j\frac{w_i\big(t-s_1+u,h_i-\big(u-s_1+s_2\big)\big)}{l\big(t-s_1+u\big)} du.
\end{align*} }

Hence we have the following
{
\begin{align} 
&\left|w_i(t,h_i-s_2)-w_i(t,h_i-s_1)\right|\nonumber\\
\leq& \left|w_i\big(t-s_2,h_i\big)- w_i\big(t-s_1,h_i\big)\right|
+\left|\int_{s_1-s_2}^0-2\sum_{h_j} p_j\frac{w_i\big(t-s_1+u,h_i-\big(u-s_1+s_2\big)\big)}{l\big(t-s_1+u\big)}du\right|\nonumber\\
&+\int_{0}^{s_1}2\sum_{h_j\leq h_i-u}p_j \left|\frac{w_i\big(t-s_1+u,h_i-u\big)-w_i\big(t-s_1+u,h_i-\big(u-s_1+s_2\big)\big)}{l\big(t-s_1+u\big)}\right|du\nonumber\\
&+2\int_{0}^{s_1} 
\left|
\sum_{h_j\leq h_i-(u-s_1+s_2)} \frac{p_j w_i\big(t-s_1+u,h_i-\big(u-s_1+s_2\big)\big)}{l\big(t-s_1+u\big)} \right.\nonumber\\
&\quad\quad\quad\quad\quad\quad\quad\left.-\sum_{h_j\leq h_i-u} \frac{p_j w_i\big(t-s_1+u,h_i-\big(u-s_1+s_2\big)\big)}{l\big(t-s_1+u\big)} 
\right|du .
\label{kpw_12_0}
\end{align} }

For the first term on the RHS of (\ref{kpw_12_0}),
\begin{align}
&\left|w_i\big(t-s_2,h_i\big)- w_i\big(t-s_1,h_i\big)\right|\nonumber\\
&\leq\left|2p_i\frac{f\big(t-s_2\big)}{l\big(t-s_2\big)}-2p_i\frac{f\big(t-s_1\big)}{l\big(t-s_1\big)}\right|
+\sum_{j>i}\int_{ h_i}^{h_j}2p_i
\left|
\frac{w_j\big(t-s_1,u\big)}{l\big(t-s_1\big)}-\frac{w_j\big(t-s_2,u\big)}{l\big(t-s_2\big)}
\right| du
\nonumber\\
&\leq2\left[\frac{f\big(t-s_2\big)-f\big(t-s_1\big)}{l\big(t-s_2\big)}+\frac{f\big(t-s_1\big)}{l\big(t-s_1\big)}\frac{l\big(t-s_1\big)-l\big(t-s_2\big)}{l\big(t-s_2\big)}\right]+..\nonumber.\\
&\leq \frac{6\epsilon}{m}+\frac{2(3+4M h_M \xi)\epsilon}{m}
+\sum_{j>i}\int_{ h_i}^{h_j}2p_i
\left|
\frac{w_j\big(t-s_1,u\big)}{l\big(t-s_1\big)}-\frac{w_j\big(t-s_2,u\big)}{l\big(t-s_2\big)}
\right| du\nonumber\\
&\leq\frac{6\epsilon}{m}+\frac{2\big(3+4M h_M \xi\big)\epsilon}{m}+2h_M\xi \frac{2\big(3+4M h_M\xi\big)\epsilon}{m}\nonumber\\
&\quad+\sum_{j>i}\int_{ h_i}^{h_j}2p_i
\left|
w_j\big(t-s_1,u\big)-w_j\big(t-s_2,u\big)
\right|
du\label{kpw_12_1},
\end{align}
where we use the facts $f(t)<l(t)$, $s_2-s_1<\epsilon$, $l(t)>m$, and $\xi>w_j$. For the second term on the RHS of (\ref{kpw_12_0}),
\begin{align}
\left|\int_{s_1-s_2}^0-2 \sum_{h_j}p_j \frac{w_i\big(t-s_1+u,h_i-\big(u-s_1+s_2\big)\big)}{l\big(t-s_1+u\big)}du\right|\leq\int_{0}^{\epsilon}2\xi
\leq 2\xi\epsilon\label{kpw_12_2},
\end{align}
where $\xi$ is an upper bound for $\frac{w_2(s+u,u)}{l(s+u)}$ for all $s,u$. For the last term on the RHS of (\ref{kpw_12_0}), because $|s_1-s_2|<\epsilon$, we have
{
\begin{align} 
&2\int_{0}^{s_1} 
\left|
\sum_{h_j\leq h_i-(u-s_1+s_2)} p_j\frac{w_2\big(t-s_1+u,h_i-\big(u-s_1+s_2\big)\big)}{l\big(t-s_1+u\big)}\right.\nonumber\\
&\left.\quad\quad\quad\quad\quad
-\sum_{h_j\leq h_i-u} p_j\frac{w_2\big(t-s_1+u,h_i-\big(u-s_1+s_2\big)\big)}{l\big(t-s_1+u\big)} 
\right|
du\nonumber\\
&\leq 2\int_{0}^{s_1} 
\left|
\sum_{h_i-u <h_j\leq h_i-\big(u-s_1+s_2\big)} p_j\frac{w_2\big(t-s_1+u,h_i-\big(u-s_1+s_2\big)\big)}{l\big(t-s_1+u\big)} 
\right| 
du\nonumber\\
&\leq 2\sum_{j<i}\int_{h_i-h_j+s_2-s_i}^{h_i-h_j}
\left|
p_j\frac{w_2\big(t-s_1+u,h_i-\big(u-s_1+s_2\big)\big)}{l\big(t-s_1+u\big)} 
\right| 
du\leq 2 \xi \epsilon \label{kpw_12_4}.
\end{align}}

All that left is the third term in (\ref{kpw_12_0}). Define
\begin{align}
y_i(u)=\left|w_i\big(t-s_1+u,h_i-u\big)-w_i\big(t-s_1+u,h_i-\big(u-s_1+s_2\big)\big)\right|\label{kpw_12_5}.
\end{align}

Then we have $\left|w_i(t,h_i-s_2)-w_i(t,h_i-s_1)\right|=y_i(s_1)$ and substituting (\ref{kpw_12_1}), (\ref{kpw_12_2}), (\ref{kpw_12_4}), and (\ref{kpw_12_5}) into (\ref{kpw_12_0}) we get
\begin{align*} 
y_i(s_1)&\leq 4\xi\epsilon+\frac{6\epsilon}{m}+\frac{2(3+4M h_M\xi)\epsilon}{m}
+2h_M\xi \frac{2(3+4M h_M \xi)\epsilon}{m}\\
&+\sum_{j>i}\int_{ h_i}^{h_j}2p_i
\left|
w_j\big(t-s_1,u\big)-w_j\big(t-s_2,u\big)
\right|
du+\frac{2}{m}\int_{0}^{s_1}y_i(u)du.
\end{align*} 
And since
{
\begin{align*} 
&\left|
w_j\big(t-s_1,u\big)-w_j\big(t-s_2,u\big)
\right|\\
\leq& 
\left|
w_j\big(t-s_1,u\big)-w_j\big(t-s_1,u-s_2+s_1\big)
\right|\nonumber\\
&+2\left|
\int_{0}^{s_1-s_2}\sum_{h_r\leq u-s_2+s_1+s}p_j\frac{w_j\big(t-s_1+s,u-s_2+s_1+s\big)}{l\big(t-s_2+s\big)}du
\right|\nonumber\\
\leq&\left|
w_j\big(t-s_1,u\big)-w_j\big(t-s_1,u-s_2+s_1\big)
\right|+2\xi\epsilon,
\end{align*}}
we get
\begin{align} 
&y_i(s_1)\leq 4\xi\epsilon+\frac{6\epsilon}{m}+\frac{2(3+4M h_M \xi)\epsilon}{m}
+2h_M\xi \frac{2(3+4M h_M \xi)\epsilon}{m}\nonumber\\
&+2Mh_M\xi\epsilon+\sum_{j>i}\int_{ h_i}^{h_j}2p_i
\left|
w_j\big(t-s_1,u\big)-w_j\big(t-s_1,u-s_2+s_1\big)
\right|
du+\frac{2}{m}\int_{0}^{s_1}y_i(u)du\label{kpw_12_6}.
\end{align} 
We now use induction on $i$ to (\ref{kpw_12_6}).

\textbf{Case 1} : $i=M$. There is no term in the summation $\sum_{j>i}$ included in (\ref{kpw_12_6}). So 
\begin{align*} 
y_i\big(s_1\big)\leq& (2+2Mh_M)\xi\epsilon+\frac{6\epsilon}{m}+\frac{2(3+4 M h_M \xi)\epsilon}{m}
+\frac{2(3+4M h_M \xi)\epsilon}{m}\\
&+2h_M\xi \frac{2(3+4 M h_M \xi)\epsilon}{m}
+\frac{2}{m}\int_{0}^{s_1}y(u)du\\
=:&\alpha_{1,0}\epsilon+\frac{2}{m}\int_{0}^{s_1}y_i(u)du
\end{align*} 
Applying Gronwall's inequality we get
\begin{align*} 
y_M(s_i)=\left|w_M(t,h_M-s_2)-w_M(t,h_M-s_1)\right|\leq\alpha_{1,0}\epsilon \exp\left\{\int_{0}^{s_1}\frac{2}{m}du\right\}\leq\alpha_{1,0}\epsilon \exp\left\{\frac{2h_M}{m}\right\}.
\end{align*} 

\textbf{Case 2}: $i=M-1$.
Using the result for case 1, we get
\begin{align*} 
y_i(s_1)\leq& \alpha_{1,0}\epsilon
+\sum_{j>i}\int_{ h_i}^{h_j}2p_i
\left|
w_j\big(t-s_1,u\big)-w_j\big(t-s_2,u\big)
\right|
du+\frac{2}{m}\int_{0}^{s_1}y_i(u)du\\
\leq& \alpha_{1,0}\epsilon
+\int_{ h_i}^{h_M}2
\left|
w_M\big(t-s_1,u\big)-w_M\big(t-s_2,u\big)
\right|
du +\frac{2}{m}\int_{0}^{s_1}y_i(u)du\\
\leq& \alpha_{1,0}\epsilon+2h_M\alpha_{1,0} \exp\left\{\frac{2h_M}{m}\right\}\epsilon+\frac{2}{m}\int_{0}^{s_1}y_i(u)du\\
\leq&(3h_M)\alpha_{1,0} \exp\left\{\frac{2h_M}{m}\right\}\epsilon+\frac{2}{m}\int_{0}^{s_1}y_i(u)du.
\end{align*} 

Again applying Gronwall's inequality we get
\begin{align*} 
\left|w_{M-1}(t,s_2)-w_{M-1}(t,s_1)\right|\leq(2h_M+1)\alpha_{1,0} \exp\left\{\frac{4h_M}{m}\right\}\epsilon. 
\end{align*} 
Repeat and iterate on $i$ we get for all $i\leq M$
we get 
\begin{align*} 
\left|w_i(t,s_2)-w_i(t,s_1)\right|\leq \alpha_{0,1}M\big(h_M\big)^M\exp\left\{\frac{2Mh_M}{m}\right\}\epsilon. 
\end{align*} 
\end{proof}

%\tempthirteen*
\begin{proof}[Proof of Lemma \ref{k_tempthirteen}]
Using Lemmas \ref{k_dJ1} and \ref{k_extendw}, we get
\begin{align} 
&  N^{-1}\left|
\sum_{j>i} \sum_{u=h_i}^{h_j-\epsilon}J_{j,i}(t_{n-k},u)-N
\sum_{j>i} \int_{h_i}^{h_j}2p_i\frac{w_j(t_{n-k},u)}{l(t_{n-k})}du
\right|\nonumber\\
&\leq \theta+\theta_1+  N^{-1}\left|
\sum_{j>i} \sum_{u=h_i}^{h_j-\epsilon}E[J_{j,i}(t_{n-k},u)|\sigma_{n-k}]-N
\sum_{j>i} \int_{h_i}^{h_j}2p_i\frac{w_j(t_{n-k},u)}{l(t_{n-k})}du
\right|\nonumber\\
&\leq \theta+\theta_1+ N^{-1}\left|
\sum_{j>i} \sum_{u=h_i}^{h_j-\epsilon}\frac{2p_i N W_j(t_{n-k},u)}{L(t_{n-k})}-N
\sum_{j>i} \int_{h_i}^{h_j}2p_i \frac{w_j(t_{n-k},u)}{l(t_{n-k})}du
\right|\nonumber\\
&\quad+ N^{-1}(M\cdot h_M\epsilon^{-1}\cdot 32\epsilon^3)\nonumber\\
&\leq \theta+\theta_1
+ N^{-1}(M\cdot h_M\epsilon^{-1}\cdot 32\epsilon^3)+2\alpha_{0,1}M^2(h_M)^{M+1}\exp(2Mh_M/m)\epsilon\nonumber\\
& \quad +N^{-1}\left|
2p_i N \sum_{j>i} \sum_{u=h_i}^{h_j-\epsilon} 
\left(
\frac{ W_j(t_{n-k},u)}{L(t_{n-k})}-\epsilon\frac{ w_j(t_{n-k},u)}{l(t_{n-k})}
\right)
\right|\nonumber\\
&\leq \theta+\theta_1
+32Mh_M\epsilon^{2}N^{-1}+2\alpha_{0,1}M^2(h_M)^{M+1}\exp(2Mh_M/m)\epsilon\nonumber\\
&\quad+2p_i \epsilon   \sum_{j>i} \sum_{u=h_i}^{h_j-\epsilon} 
\left(
\frac{ |W_j(t_{n-k},u)/N- w_j(t_{n-k},u)|+2|L_{\lambda}(t_{n-k})-l(t_{n-k})|}{m}
\right)
\nonumber\\
&\leq \theta+\theta_1
+32Mh_M\epsilon^{2}N^{-1}+2\alpha_{0,1}M^2(h_M)^{M+1}\exp(2Mh_M/m)\epsilon\nonumber\\
&+2p_i\epsilon  \sum_{j>i} \sum_{u=h_i}^{h_j-\epsilon} 
\left(
\frac{ |W_j(t_{n-k},u)/N- w_j(t_{n-k},u)|}{m}
\right)
+\frac{4p_i}{m}Mh_M|L_{\lambda}(t_{n-k})-l(t_{n-k})|.\nonumber
\end{align}
\end{proof}

%\nocite{*}
\bibliographystyle{APT}
\footnotesize
\bibliography{b}

\end{document}